\documentclass{amsart}

\usepackage{amsfonts}
\usepackage{amsmath}
\usepackage{amssymb}
\usepackage{latexsym}
\usepackage{amscd}
\usepackage{amsthm}
\usepackage{subfigure}
\usepackage{graphicx}
\usepackage[all]{xy}
\usepackage{color}

\newtheorem{thm}{Theorem}[section]
\newtheorem{prop}[thm]{Proposition}
\newtheorem{lem}[thm]{Lemma}

\theoremstyle{definition}

\newtheorem{rem}[thm]{Remark}

\newcommand{\Z}{\mathbb{Z}}
\newcommand{\M}{\mathcal{M}}
\newcommand{\T}{\mathcal{T}}
\newcommand{\C}{\mathcal{C}}

\newcommand{\F}{\mathcal{F}}
\newcommand{\PP}{\mathcal{P}_\ast}
\newcommand{\CP}{\mathcal{P}_\otimes}

\begin{document}

\title[An infinite presentation for $\T(N_{g,n})$]
{An infinite presentation for the twist subgroup of the mapping class group of a compact non-orientable surface}

\author[$^\ast$R. Kobayashi]{$^\ast$Ryoma Kobayashi}
\address[$^\ast$R. Kobayashi]{
Department of General Education,\endgraf
National Institute of Technology, Ishikawa College,\endgraf
Tsubata, Ishikawa, 929-0392, Japan
}
\email{kobayashi\_ryoma@ishikawa-nct.ac.jp}

\author[G. Omori]{Genki Omori}
\address[G. Omori]{
Department of Mathematics,\endgraf
Faculty of Science and Technology,\endgraf
Tokyo University of Science,\endgraf
Noda, Chiba, 278-8510, Japan
}
\email{omori\_genki@ma.noda.tus.ac.jp}

\subjclass[2010]{Primary 20F05, Secondary 57M07.}
\keywords{mapping class group, twist subgroup, presentation.}
\thanks{The first author was supported by JSPS KAKENHI Grant Number JP19K14542 and 22K13920.
The second author was supported by JSPS KAKENHI Grant Number JP19K23409 and JP21K13794.}
\thanks{Data sharing not applicable to this article as no datasets were generated or analysed during the current study.}

\maketitle

\begin{abstract}
A finite presentation for the subgroup of the mapping class group of a compact non-orientable surface generated by all Dehn twists was given by Stukow~\cite{St4}.
In this paper, we give an infinite presentation for this group, mainly using the presentation given by Stukow~\cite{St4} and Birman exact sequences on mapping class groups of non-orientable surfaces.
\end{abstract}

\section{Introduction}

\subsection{Background}\label{BG}\

For $g\geq1$ and $n\geq0$, let $N_{g,n}$ denote a compact non-orientable surface of genus $g$ with $n$ boundary components, that is, $N_{g,n}$ is a surface obtained by removing $n$ open disks from a connected sum of $g$ real projective planes.
For $g\geq0$ and $n\geq0$, let $\Sigma_{g,n}$ denote a compact orientable surface of genus $g$ with $n$ boundary components, that is, $\Sigma_{g,n}$ is a surface obtained by removing $n$ open disks from a connected sum of $g$ tori.
As shown in Figure~\ref{non-ori-surf}, we can regard $N_{g,n}$ as a surface obtained by attaching $g-2h$ M\"obius bands to $g-2h$ boundary components of $\Sigma_{h,n+g-2h}$, for $\displaystyle0\leq{h}<\frac{g}{2}$.
We call these attached M\"obius bands \textit{crosscaps}.

\begin{figure}[htbp]
\includegraphics{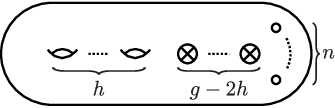}
\caption{A model of a non-orientable surface $N_{g,n}$.}\label{non-ori-surf}
\end{figure}

The mapping class group of $N_{g,n}$, denoted by $\M(N_{g,n})$, is the group consisting of isotopy classes of all diffeomorphisms of $N_{g,n}$ which fix the boundary pointwise.
The mapping class group of $\Sigma_{g,n}$, denoted by $\M(\Sigma_{g,n})$, is the group consisting of isotopy classes of all \textit{orientation-preserving} diffeomorphisms of $\Sigma_{g,n}$ which fix the boundary pointwise.
It is well known that $\M(\Sigma_{g,n})$ is generated by only \textit{Dehn twists} (see \cite{D1, Li2, D2}).
On the other hand, $\M(N_{g,n})$ can not be generated by only Dehn twists.
As generators of $\M(N_{g,n})$, other than Dehn twists, \textit{crosscap slides} or \textit{crosscap transpositions} are needed (see \cite{Li1, Li3}).
Let us consider the subgroup of $\M(N_{g,n})$ generated by all Dehn twists, denoted by $\T(N_{g,n})$.
We call $\T(N_{g,n})$ the \textit{twist subgroup} of $\M(N_{g,n})$.

We now explain about the history of studies on presentations for $\M(\Sigma_{g,n})$, $\M(N_{g,n})$ and $\T(N_{g,n})$.
Finite presentations for $\M(\Sigma_{g,n})$ were given by Hatcher-Thurston~\cite{HT} and Harer~\cite{H}, and subsequently simplified by Wajnryb~\cite{W} and Matsumoto~\cite{M} for $n\leq1$.
Gervais~\cite{G2} and Labru\`ere-Paris~\cite{LP} gave finite presentations of $\M(\Sigma_{g,n})$ for $n\geq2$.
Gervais~\cite{G1} gave an infinite presentation for $\M(\Sigma_{g,n})$ by using the presentation for $\M(\Sigma_{g,n})$ given in \cite{H, W}, and then Luo~\cite{Lu} simplified its presentation.
Finite presentations for $\M(N_{2,0})$, $\M(N_{2,1})$, $\M(N_{3,0})$ and $\M(N_{4,0})$ ware given by \cite{Li1}, \cite{St1}, \cite{BC} and \cite{Sz} respectively.
Note that $\M(N_{1,0})$ and $\M(N_{1,1})$ are trivial (see \cite{E}).
Paris-Szepietowski~\cite{PS} gave a finite presentation of $\M(N_{g,n})$ for $g+n>3$ with $n\leq1$.
Stukow~\cite{St3} gave another finite presentation of $\M(N_{g,n})$ for $g+n>3$ with $n\leq1$, applying Tietze transformations for the presentation of $\M(N_{g,n})$ given in \cite{PS}.
The second author~\cite{O} gave an infinite presentation of $\M(N_{g,n})$ for $g\geq1$ and $n\leq1$, using the presentation of $\M(N_{g,n})$ given in \cite{St3}, and then, following this work, the authors~\cite{KO} gave an infinite presentation of $\M(N_{g,n})$ for $g\geq1$ and $n\geq2$.
It is known that $\T(N_{g,n})$ is the index $2$ subgroup of $\M(N_{g,n})$ (see \cite{Li3}).
Stukow~\cite{St4} gave a finite presentation of $\T(N_{g,n})$ for $g+n>3$ with $n\leq1$, applying the Reidemeister-Schreier method for the presentation of $\M(N_{g,n})$ given in \cite{St3} (see Theorems~\ref{FPT1} and \ref{FPT0}).

In this paper, we give an infinite presentation of $\T(N_{g,n})$ for $g\geq1$ and $n\geq0$ (see Theorem~\ref{main-thm}), mainly using the presentation of  $\T(N_{g,n})$ given in \cite{St4} and Birman exact sequences on mapping class groups of non-orientable surfaces.

Through this paper, the product $gf$ of mapping classes $f$ and $g$ means that we apply $f$ first and then $g$.
Moreover we do not distinguish a loop from its isotopy class.

\subsection{Main result}\label{MR}\

For a simple closed curve $c$ of $N_{g,n}$, a regular neighborhood of $c$ is either an annulus  or a M\"obius band.
We call $c$ a \textit{two sided} or a \textit{one sided} simple closed curve respectively.
For a two sided simple closed curve $c$, we can take two orientations $+_c$ and $-_c$ of a regular neighborhood of $c$.
The \textit{right handed Dehn twist} $t_{c;\theta}$ about a two sided simple closed curve $c$ with respect to $\theta\in\{+_c,-_c\}$ is the isotopy class of the map described as shown in Figure~\ref{dehn}.
$t_{c;\theta}$ does not depend on a choice of a representative curve of the isotopy class of $c$ and its regular neighborhood.
We remark that although the Dehn twist was defined for an \textit{oriented} simple closed curve in \cite{O, KO}, in this paper we do not consider an orientation of a simple closed curve in the definition of the Dehn twist.
We write $t_{c;\theta}=t_c$ if the orientation $\theta$ is given explicitly. That is, the direction of the twist is indicated by an arrow written beside $c$ as shown in Figure~\ref{dehn}.

\begin{figure}[htbp]
\includegraphics{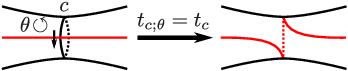}
\caption{The Dehn twist $t_{c;\theta}=t_{c}$ about $c$ with respect to $\theta\in\{+_\alpha,-_\alpha\}$.}\label{dehn}
\end{figure}

We denote by $f_\ast(\theta)$ the orientation of a regular neighborhood of $f(c)$ induced from $\theta\in\{+_c,-_c\}$, for a two sided simple closed curve $c$ of $N_{g,n}$ and $f\in\M(N_{g,n})$.
Let $c_1,\dots,c_k$, $c_0$, $c_0^\prime$ and $d_1,\dots,d_7$ be simple closed curves with arrows as shown in Figure~\ref{chain-lantern}.
$\M(N_{g,n})$ admits following relations.
\begin{enumerate}
\item\label{trivial-1}	$t_{c;\theta}=1$ if $c$ bounds a disk or a M\"obius band.
\item\label{trivial-2}	$t_{c;-_c}^{-1}=t_{c;+_c}$.
\item\label{conj}	$ft_{c;\theta}f^{-1}=t_{f(c);f_\ast(\theta)}$ for $f\in\M(N_{g,n})$.
\item\label{k-chain}	$(t_{c_1}t_{c_2}\cdots{}t_{c_k})^{k+1}=t_{c_0}t_{c_0^\prime}$ if $k$ is odd,\\
				$(t_{c_1}t_{c_2}\cdots{}t_{c_k})^{2k+2}=t_{c_0}$ if $k$ is even.
\item\label{lantern}	$t_{d_1}t_{d_2}t_{d_3}=t_{d_4}t_{d_5}t_{d_6}t_{d_7}$.
\end{enumerate}
We can check the relations (\ref{trivial-1}) and (\ref{trivial-2}) easily.
The relations (\ref{conj}), (\ref{k-chain}) and (\ref{lantern}) are called a \textit{conjugation relation}, a \textit{$k$-chain relation}, a \textit{lantern relation} respectively.
These are famous relations on mapping class groups.
In the relation (\ref{conj}), if $f=t_{c^\prime;\theta^\prime}$, $|c\cap{}c^\prime|=0$ or $1$, and the orientations $\theta$ and $\theta^\prime$ are compatible, then the relation can be rewritten as a \textit{commutativity relation} $t_{c;\theta}t_{c^\prime;\theta^\prime}=t_{c^\prime;\theta^\prime}t_{c^;\theta}$ or a \textit{braid relation} $t_{c;\theta}t_{c^\prime;\theta^\prime}t_{c;\theta}=t_{c^\prime;\theta^\prime}t_{c^;\theta}t_{c^\prime;\theta^\prime}$ respectively.

\begin{figure}[htbp]
\includegraphics{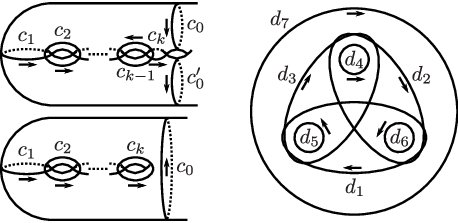}
\caption{}\label{chain-lantern}
\end{figure}

Our main result is as follows.

\begin{thm}\label{main-thm}
For $g\geq1$ and $n\geq0$, $\T(N_{g,n})$ admits a presentation with a generating set
$$X=\left\{t_{c;\theta}
\left|
\begin{array}{l}
c~\textrm{is a two sided simple closed curve of}~N_{g,n}~\textrm{and}\\
\theta~\textrm{is an orientation of a regular neighborhood of}~c.
\end{array}
\right.
\right\}.$$
The defining relations are
\begin{enumerate}
\item\label{bound}	$t_{c;\theta}=1$ if $c$ bounds a disk or a M\"obius band,
\item\label{inverse}	$t_{c;-_c}^{-1}=t_{c;+_c}$,
\item\label{t-conj}	all the conjugation relations $ft_{c;\theta}f^{-1}=t_{f(c);f_\ast(\theta)}$ for $f\in{X}$,
\item\label{2-chain}	all the $2$-chain relations,
\item\label{lanterns}	all the lantern relations,
\end{enumerate}
\end{thm}

\begin{rem}\label{(2,0)}
If a two sided simple closed curve $c$ is a non-separating curve of $N_{2,0}$, then we can check that $t_{c;+_c}$ and $t_{c;-_c}$ are the same elements in $X$ of Theorem~\ref{main-thm}.
Hence by the relation (\ref{inverse}) of Theorem~\ref{main-thm}, we have the relation $t_{c;\theta}^2=1$.
\end{rem}

\begin{rem}\label{chain-rel}
Applying Theorem in \cite{Lu} to a regular neighborhood of $\cup_{i=1}^kc_i$ in Figure~\ref{chain-lantern}, it follows that any chain relation is obtained from the relations (\ref{bound}), (\ref{t-conj}), (\ref{2-chain}) and (\ref{lanterns}) of Theorem~\ref{main-thm}.
We assign the label $(4)^\prime$ to all the chain relations.
\end{rem}

Here is the outline of this paper.
In Section~\ref{BM}, we explain what we need to prove our main result.
In Section~\ref{main-thm-low-n}, we prove that the infinitely presented group with the presentation in Theorem~\ref{main-thm} is isomorphic to $\T(N_{g,n})$ with a presentation which is already known, for $n\leq1$.
In Section~\ref{main-thm-high-n}, we complete the proof of Theorem~\ref{main-thm}, by induction on $n$.

\section{Preliminaries}\label{BM}

In this section, we explain what we need to prove our main result Theorem~\ref{main-thm}.
In Section~\ref{CPFB}, we define the \textit{capping map}, the \textit{point pushing map}, the \textit{forgetful map} and the \textit{crosscap pushing map} on mapping class groups for non-orientable surfaces.
In Section~\ref{BT}, we introduce a finite presentation of $\T(N_{g,n})$ for $g\geq1$ and $n\leq1$.

\subsection{Homomorphisms on mapping class groups of non-orientable surfaces}\label{CPFB}\

We first define the \textit{capping map}, the \textit{point pushing map} and the \textit{forgetful map} on mapping class groups of non-orientable surfaces.
Take a point $\ast$ in the interior of $N_{g,n-1}$.
Let $\M(N_{g,n-1},\ast)$ denote the group consisting of isotopy classes of all diffeomorphisms of $N_{g,n-1}$ which fix $\ast$ and the boundary pointwise.
We can regard $N_{g,n}$ as a complement of an open disk neighborhood of $\ast$ in $N_{g,n-1}$.
The natural embedding $N_{g,n}\hookrightarrow{}N_{g,n-1}$ induces the homomorphism
$$\C:\M(N_{g,n})\to\M(N_{g,n-1},\ast)$$
which is called the \textit{capping map}.
The \textit{point pushing map}
$$\PP:\pi_1(N_{g,n-1},\ast)\to\M(N_{g,n-1},\ast)$$
is defined as follows.
For any loop $x\in\pi_1(N_{g,n-1},\ast)$, $\PP(x)$ is described by pushing $\ast$ once along $x$.
The \textit{forgetful homomorphism}
$$\F:\M(N_{g,n-1},\ast)\to\M(N_{g,n-1})$$
is defined naturally, that is, $\F(f)$ does not necessarily fix $\ast$.

Next, we consider exact sequences on mapping class groups of non-orientable surfaces.
Let $\M^+(N_{g,n-1},\ast)$ denote the subgroup of $\M(N_{g,n-1},\ast)$ consisting of elements which preserve a local orientation of $\ast$, and $\pi_1^+(N_{g,n-1},\ast)$ the subgroup of $\pi_1(N_{g,n-1},\ast)$ generated by two sided simple loops.
We have the exact sequences
\begin{eqnarray}
1\to\Z\to\M(N_{g,n})\overset{\C}{\to}\M^+(N_{g,n-1},\ast)\to1,\label{MCES}\\
\pi_1^+(N_{g,n-1},\ast)\overset{\PP}{\to}\M^+(N_{g,n-1},\ast)\overset{\F}{\to}\M(N_{g,n-1})\to1.\label{MBES}
\end{eqnarray}
The second sequence is called the \textit{Birman exact sequence}, introduced by Birman~\cite{B}.
Note that $\PP$ is injective except for the case $(g,n)=(2,1)$.
For details see \cite{B, Ko, St2, FM, PS}.

\begin{rem}\label{PP1}
For a simple loop $\alpha\in\pi_1^+(N_{g,n-1},\ast)$, we take an orientation $\theta$ of a regular neighborhood of $\alpha$.
Let $\gamma_1$ and $\gamma_2$ be the right side boundary curve and the left side boundary curve of the regular neighborhood of $\alpha$ for the direction of $\alpha$, respectively, where the left and right sides are determined by $\theta$.
Then we have $\PP(\alpha)=t_{\gamma_1;\theta_1}t_{\gamma_2;\theta_2}^{-1}$, where $\theta_1$ and $\theta_2$ are the orientations compatible with $\theta$ (see Figure~\ref{PP}).

\begin{figure}[htbp]
\includegraphics{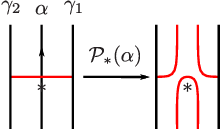}
\caption{The boundary curves $\gamma_1$ and $\gamma_2$ of a regular neighborhood of $\alpha$.}\label{PP}
\end{figure}

\end{rem}

\begin{rem}\label{PP2}
Let $\alpha$, $\beta$ and $\gamma\in\pi_1^+(N_{g,n-1},\ast)$ be simple loops with $\alpha\beta=\gamma$.
If $\alpha$ and $\beta$ intersect transversally at only $\ast$ as shown in Figure~\ref{alpha-beta}, then the relation $\PP(\gamma)=\PP(\beta)\PP(\alpha)$ is obtained from the relation (\ref{t-conj}) of Theorem~\ref{main-thm} (see Lemma~5.4 in \cite{KO}).
If $\alpha$ and $\beta$ intersect tangentially at only $\ast$ as shown in Figure~\ref{alpha-beta}, then the relation $\PP(\gamma)=\PP(\beta)\PP(\alpha)$ is obtained from the relation (\ref{lanterns}) of Theorem~\ref{main-thm} and a relation $t_{c;\theta}=1$, where $c$ is a simple closed curve bounding a disk neighborhood of $\ast$ (see Lemma~5.2 in \cite{KO}).
We call this relation the \textit{extended lantern relation}.

\begin{figure}[htbp]
\includegraphics{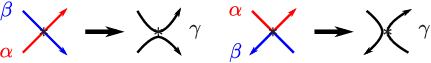}
\caption{}\label{alpha-beta}
\end{figure}

\end{rem}

Finally, we define the \textit{crosscap pushing map}.
Let $S$ be a surface obtained by shrinking some crosscap of $N_{g,n}$ into a point.
We call this operation the \textit{blowdown} with respect to the crosscap.
Note that $S$ is diffeomorphic to either $N_{g-1,n}$ or $\Sigma_{\frac{g-1}{2},n}$ (see Figure~\ref{blowup-down}).
Let $\ast$ be the shrinked point of $S$.
Conversely, we can obtain $N_{g,n}$ from $S$, and call this operation the \textit{blowup} with respect to $\ast$.
The \textit{crosscap pushing map}
$$\CP:\pi_1(S,\ast)\to\M(N_{g,n})$$
is defined as follows.
For $x\in\pi_1(S,\ast)$, let $\tilde{x}$ be an oriented loop of $N_{g,n}$ induced from $x$ by the blowup with respect to $\ast$.
$\CP(x)$ is described by pushing the crosscap, which is obtained by the blowup with respect to $\ast$, once along $\tilde{x}$ (see Figure~\ref{CP}).

\begin{figure}[htbp]
\includegraphics{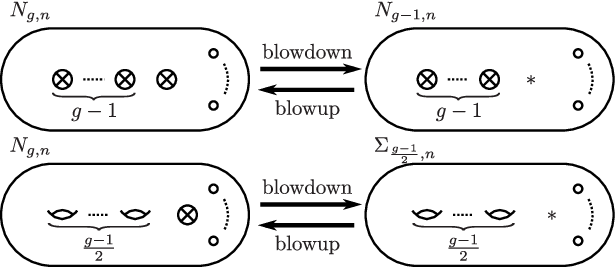}
\caption{}\label{blowup-down}
\end{figure}

\begin{figure}[htbp]
\includegraphics{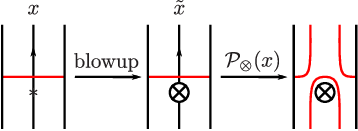}
\caption{The crosscap pussing map $\CP$.}\label{CP}
\end{figure}

Similar to Remarks~\ref{PP1} and \ref{PP2}, we have the followings.

\begin{rem}\label{CP1}
For a two sided simple loop $\alpha\in\pi_1(S,\ast)$, we take an orientation $\theta$ of a regular neighborhood of $\alpha$.
Let $\gamma_1$ and $\gamma_2$ be the right side boundary curve and the left side boundary curve of the regular neighborhood of $\alpha$ for the direction of $\alpha$, respectively, where the left and right sides are determined by $\theta$.
We also denote by $\tilde{\gamma}_1$ and $\tilde{\gamma}_2$ loops of $N_{g,n}$ induced from $\gamma_1$ and $\gamma_2$ by the blowup with respect to $\ast$, respectively.
Then we have $\CP(\alpha)=t_{\tilde{\gamma}_1;\tilde{\theta}_1}t_{\tilde{\gamma}_2;\tilde{\theta}_2}^{-1}$, where $\tilde{\theta}_1$ and $\tilde{\theta}_2$ are the orientations of regular neighborhoods of $\tilde{\gamma}_1$ and $\tilde{\gamma}_2$ induced by the blowup with respect to $\ast$ from the orientations of regular neighborhoods of $\gamma_1$ and $\gamma_2$ which are compatible with $\theta$, respectively.
\end{rem}

\begin{rem}\label{CP2}
Let $\alpha$, $\beta$ and $\gamma\in\pi_1(S,\ast)$ be two sided simple loops with $\alpha\beta=\gamma$.
If $\alpha$ and $\beta$ intersect transversally at only $\ast$ as shown in Figure~\ref{alpha-beta}, then the relation $\CP(\gamma)=\CP(\beta)\CP(\alpha)$ is obtained from the relarions (\ref{t-conj}) of Theorem~\ref{main-thm} (see Lemma~5.4 in \cite{KO}).
If $\alpha$ and $\beta$ intersect tangentially at only $\ast$ as shown in Figure~\ref{alpha-beta}, then the relation $\CP(\gamma)=\CP(\beta)\CP(\alpha)$ is obtained from the relations (\ref{bound}) and (\ref{lanterns}) of Theorem~\ref{main-thm} (see Lemma~5.2 in \cite{KO}).
\end{rem}

\subsection{A finite presentation for $\T(N_{g,n})$}\label{BT}\

For $g\geq2$ and $n\geq0$ we have the short exact sequence
\begin{eqnarray}
1\to\T(N_{g,n})\to\M(N_{g,n})\to\Z/{2\Z}\to1\label{TMZ2}
\end{eqnarray}
(see \cite{Li1,Li3}).
Using the Reidemeister Schreier method for a presentation of $\M(N_{g,n})$, we can obtain a presentation of $\T(N_{g,n})$.

For $n\leq1$, let $\alpha_1,\dots,\alpha_{g-1}$, $\beta$, $\gamma$, $\varepsilon$ and $\zeta$ be simple closed curves of $N_{g,n}$ with arrows as shown in Figure~\ref{generator}, and if $g$ is even, $\beta_0,\dots,\beta_{\frac{g-2}{2}}$, $\bar{\beta}_{\frac{g-6}{2}}$, $\bar{\beta}_{\frac{g-4}{2}}$ and $\bar{\beta}_{\frac{g-2}{2}}$ simple closed curves of $N_{g,n}$ with arrows as shown in Figure~\ref{generator}.
Let $\alpha$ be an oriented simple loop of $N_{g-1,n}$ based at $\ast$ induced from $\alpha_1$ by the blowdown with respect to the first crosscap, as shown in Figure~\ref{generator}.
For simplicity, we denote $t_{\alpha_i}=a_i$, $t_\beta=b$, $t_\gamma=c$, $t_\varepsilon=e$, $t_\zeta=f$, $t_{\beta_i}=b_i$, $t_{\bar{\beta}_i}=\bar{b}_i$ and $\CP(\alpha)=y$.
Note that $y^2=t_\delta$, where $\delta$ is a simple close curve of $N_{g,n}$ with an arrow as shown in Figure~\ref{generator}.
For $g\geq1$ and $n\leq1$, $\T(N_{g,n})$ admits a finite presentation given in Theorems~\ref{FPT1}, \ref{FPT0} or Lemma~\ref{FPT}.

\begin{figure}[htbp]
\includegraphics{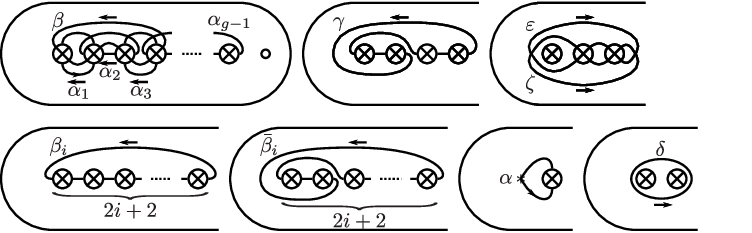}
\caption{}\label{generator}
\end{figure}

\begin{thm}[Theorem~3.1 in \cite{St4}]\label{FPT1}
If $g\geq3$ is odd or $g=4$, then $\T(N_{g,1})$ admits a presentation with generators $a_1,\dots,a_{g-1}$, $e$, $f$, $y^2$ and $b$, $c$ for $g\geq4$.
The defining relations are
\begin{itemize}
\item[(A1)]						$a_ia_j=a_ja_i$ for $g\geq4$, $|i-j|>1$,
\item[(A2)]						$a_ia_{i+1}a_i=a_{i+1}a_ia_{i+1}$ for $i=1,\dots,g-2$,
\item[(A3)]						$a_ib=ba_i$ for $g\geq4$, $i\neq4$,
\item[(A4)]						$ba_4b=a_4ba_4$ for $g\geq5$,
\item[(A5)]						$(a_2a_3a_4b)^{10}=(a_1a_2a_3a_4b)^6$ for $g\geq5$,
\item[(A6)]						$(a_2a_3a_4a_5a_6b)^{12}=(a_1a_2a_3a_4a_5a_6b)^9$ for $g\geq7$,
\item[($\overline{\mathrm{A1}}_1$)]	$ea_j=a_je$ for $g\geq5$, $j\geq4$,
\item[($\overline{\mathrm{A1}}_2$)]	$fa_j=a_jf$ for $g\geq5$, $j\geq4$,
\item[($\overline{\mathrm{A2}}_1$)]	$a_1ea_1=ea_1e$,
\item[($\overline{\mathrm{A2}}_2$)]	$a_3^{-1}ea_3^{-1}=ea_3^{-1}e$ for $g\geq4$,
\item[($\overline{\mathrm{A2}}_3$)]	$a_1fa_1=fa_1f$,
\item[($\overline{\mathrm{A3}}_1$)]	$a_1c=ca_1$ for $g=4$, $5$,
\item[($\overline{\mathrm{A3}}_2$)]	$ec=ce$ for $g=4$, $5$,
\item[($\overline{\mathrm{A4}}$)]	$ca_4c=a_4ca_4$ for $g=5$, $6$,
\item[($\overline{\mathrm{A5}}$)]	$(e^{-1}a_3a_4c)^{10}=(a_1^{-1}e^{-1}a_3a_4c)^6$ for $g=5$, $6$,
\item[($\overline{\mathrm{A6}}$)]	$(e^{-1}a_3a_4a_5a_6c)^{12}=(a_1^{-1}e^{-1}a_3a_4a_5a_6c)^9$ for $g=7$, $8$,
\item[($\overline{\mathrm{B1}}$)]	$(a_2a_3a_1a_2ea_1a_3^{-1}e)(a_2a_3a_1a_2fa_1a_3^{-1}f)=1$ for $g\geq4$,
\item[($\overline{\mathrm{B2}}_1$)]	$y^2=a_2a_1ea_1a_2a_1a_2a_1a_2fa_1a_2$,
\item[($\overline{\mathrm{B2}}_2$)]	$(a_2a_1ea_1a_2a_1a_2a_1a_2fa_1a_2)(a_2a_1fa_1a_2a_1a_2a_1a_2ea_1a_2)=1$,
\item[($\overline{\mathrm{B3}}$)]	$y^2a_3=a_3y^2$ for $g\geq4$,
\item[($\overline{\mathrm{B4}}_1$)]	$ea_2=a_2e$,
\item[($\overline{\mathrm{B4}}_2$)]	$fa_2=a_2f$,
\item[($\overline{\mathrm{B6}}_1$)]	$bc=(a_1a_2a_3f^{-1}a_3^{-1}a_2^{-1}a_1^{-1})(a_2^{-1}a_3^{-1}e^{-1}a_3a_2)$ for $g\geq4$,
\item[($\overline{\mathrm{B6}}_2$)]	$c(y^2by^{-2})=(a_1^{-1}e^{-1}a_3a_2a_3^{-1}ea_1)(ea_3^{-1}y^2a_2y^{-2}a_3e^{-1})$ for $g=4$, $5$,
\item[($\overline{\mathrm{B7}}_1$)]	$(a_4a_5a_3a_4a_2a_3a_1a_2ea_1a_3^{-1}ea_4^{-1}a_3^{-1}a_5^{-1}a_4^{-1})c\\=b(a_4a_5a_3a_4a_2a_3a_1a_2ea_1a_3^{-1}ea_4^{-1}a_3^{-1}a_5^{-1}a_4^{-1})$ for $g\geq6$,
\item[($\overline{\mathrm{B7}}_2$)]	$(a_2^{-1}a_1^{-1}a_3^{-1}a_2^{-1}a_4^{-1}a_3^{-1}a_5^{-1}a_4^{-1})b(a_4a_5a_3a_4a_2a_3a_1a_2)y^2\\=y^2(a_2^{-1}a_1^{-1}a_3^{-1}a_2^{-1}a_4^{-1}a_3^{-1}a_5^{-1}a_4^{-1})b(a_4a_5a_3a_4a_2a_3a_1a_2)$ for $g\geq6$,
\item[($\overline{\mathrm{B8}}_1$)]	$(a_1ea_3^{-1}a_4^{-1}ca_4a_3e^{-1}a_1^{-1})(a_1^{-1}a_2^{-1}a_3^{-1}a_4^{-1}b^{-1}a_4a_3a_2a_1)\\=a_4^{-1}(a_3^{-1}a_2^{-1}e^{-1}a_3a_4a_3^{-1}ea_2a_3)a_2^{-1}e^{-1}$ for $g\geq5$,
\item[($\overline{\mathrm{B8}}_2$)]	$(a_1^{-1}a_2^{-1}a_3^{-1}a_4^{-1}ba_4a_3a_2a_1)(a_1fa_3^{-1}a_4^{-1}y^{-2}c^{-1}y^2a_4a_3f^{-1}a_1^{-1})\\=a_4^{-1}(a_3^{-1}fa_2a_3a_4a_3^{-1}a_2^{-1}f^{-1}a_3)fa_2$ for $g=5$, $6$.
\end{itemize}
If $g\geq6$ is even, then $\T(N_{g,1})$ admits a presentation with generators $a_1,\dots,a_{g-1}$, $e$, $f$, $y^2$, $b$, $c$ and additionally $b_0,\dots,b_{\frac{g-2}{2}}$, $\bar{b}_{\frac{g-6}{2}}$, $\bar{b}_{\frac{g-4}{2}}$, $\bar{b}_{\frac{g-2}{2}}$.
The defining relations are relations $(\mathrm{A1})$-$(\mathrm{A6})$, $(\overline{\mathrm{A1}}_1)$-$(\overline{\mathrm{A6}})$, $(\overline{\mathrm{B1}})$-$(\overline{\mathrm{B8}}_2)$ and additionally
\begin{itemize}
\item[(A7)]						$b_0=a_1$, $b_1=b$,
\item[(A8)]						$b_{i+1}=(b_{i-1}a_{2i}a_{2i+1}a_{2i+2}a_{2i+3}b_i)^5(b_{i-1}a_{2i}a_{2i+1}a_{2i+2}a_{2i+3})^{-6}$ for $1\leq{i}\leq\frac{g-4}{2}$,
\item[(A9a)]					$b_2b=bb_2$ for $g=6$,
\item[(A9b)]					$b_{\frac{g-2}{2}}a_{g-5}=a_{g-5}b_{\frac{g-2}{2}}$ for $g\geq8$,
\item[($\overline{\mathrm{A7a}}$)]	$\bar{b}_0=a_1^{-1}$, $\bar{b}_1=c$ for $g=6$,
\item[($\overline{\mathrm{A7b}}$)]	$\bar{b}_1=c$ for $g=8$,
\item[($\overline{\mathrm{A7c}}$)]	$\bar{b}_i=z_{g-1}b_iz_{g-1}^{-1}$ where $i=\frac{g-6}{2}$, $\frac{g-4}{2}$, $i\geq2$ and $z_{g-1}=(a_{g-1}a_ga_{g-2}a_{g-1}\cdots{}a_3a_4e^{-1}a_3a_1^{-1}e^{-1})(a_2^{-1}a_1^{-1}\cdots{}a_{g-1}^{-1}a_{g-2}^{-1}a_g^{-1}a_{g-1}^{-1})$,
\item[($\overline{\mathrm{A8a}}$)]	$\bar{b}_2=(\bar{b}_0e^{-1}a_3a_4a_5\bar{b}_1)^5(\bar{b}_0e^{-1}a_3a_4a_5)^{-6}$ for $g=6$,
\item[($\overline{\mathrm{A8b}}$)]	$\bar{b}_{\frac{g-2}{2}}=(\bar{b}_{\frac{g-6}{2}}a_{g-4}a_{g-3}a_{g-2}a_{g-1}\bar{b}_{\frac{g-4}{2}})^5(\bar{b}_{\frac{g-6}{2}}a_{g-4}a_{g-3}a_{g-2}a_{g-1})^{-6}$ for $g\geq8$,
\item[($\overline{\mathrm{A9a}}$)]	$\bar{b}_2c=c\bar{b}_2$ for $g=6$,
\item[($\overline{\mathrm{A9b}}$)]	$\bar{b}_{\frac{g-2}{2}}a_{g-5}=a_{g-5}\bar{b}_{\frac{g-2}{2}}$ for $g\geq8$.
\end{itemize}
\end{thm}

\begin{thm}[Theorem~3.2 in \cite{St4}]\label{FPT0}
If $g\geq5$ is odd, then the group $\T(N_{g,0})$ is isomorphic to the quotient of the group $\T(N_{g,1})$ with the presentation given in Theorem~\ref{FPT1} obtained by adding a generator $\varrho$ and relations
\begin{itemize}
\item[($\mathrm{C1a}$)]			$(a_1a_2\cdots{}a_{g-1})^g=\varrho$,
\item[($\overline{\mathrm{C1a}}$)]	$(a_1^{-1}e^{-1}a_3\cdots{}a_{g-1})^g=y^2\varrho$,
\item[($\mathrm{C2}$)]			$a_i\varrho=\varrho{}a_i$ for $1\leq{i}\leq{}g-1$,
\item[($\overline{\mathrm{C2}}$)]	$\varrho{}e=f\varrho$,
\item[($\overline{\mathrm{C5}}_1$)]	$\varrho{}y^2=y^{-2}\varrho$,
\item[($\mathrm{C3}$)]			$\varrho^2=1$,
\item[($\overline{\mathrm{C4a}}$)]	$(a_2a_3\cdots{}a_{g-1}e^{-1}a_3\cdots{}a_{g-1})^{\frac{g-1}{2}}=1$.
\end{itemize}
Moreover, relations $(\overline{\mathrm{A1}}_2)$, $(\overline{\mathrm{B2}}_2)$, $(\overline{\mathrm{B4}}_2)$ are superfluous.

If $g\geq4$ is even, then the group $\T(N_{g,0})$ is isomorphic to the quotient of the group $\T(N_{g,1})$ with the presentation given in Theorem~\ref{FPT1} obtained by adding a generator $\bar{\varrho}$ and relations
\begin{itemize}
\item[($\mathrm{C1b}$)]			$(a_1a_2\cdots{}a_{g-1})^g=1$,
\item[($\overline{\mathrm{C2}}_1$)]	$\bar{\varrho}a_1=a_1^{-1}\bar{\varrho}$,
\item[($\overline{\mathrm{C2}}_2$)]	$\bar{\varrho}a_i=a_i\bar{\varrho}$ for $3\leq{i}\leq{}g-1$,
\item[($\overline{\mathrm{C2}}_3$)]	$\bar{\varrho}a_2=e^{-1}\bar{\varrho}$,
\item[($\overline{\mathrm{C5}}_2$)]	$\bar{\varrho}y^2=y^{-2}\bar{\varrho}$,
\item[($\overline{\mathrm{C3}}$)]	$\bar{\varrho}^2=1$,
\item[($\overline{\mathrm{C4}}$)]	$(\bar{\varrho}a_2a_3\cdots{}a_{g-1})^{g-1}=1$.
\end{itemize}
Moreover, relations $(\overline{\mathrm{A1}}_1)$, $(\overline{\mathrm{A2}}_1)$, $(\overline{\mathrm{A2}}_2)$ are superfluous.
\end{thm}

\begin{lem}\label{FPT}
\begin{enumerate}
\item	$\T(N_{1,0})$ and $\T(N_{1,1})$ are trivial.
\item	$\T(N_{2,0})=\langle{a_1}\mid{a_1^2}\rangle$.
\item	$\T(N_{2,1})=\langle{a_1,y^2}\mid{a_1y^2a_1^{-1}y^{-2}}\rangle$.
\item	$\T(N_{3,0})=\langle{a_1,a_2}\mid{a_1a_2a_1a_2^{-1}a_1^{-1}a_2^{-1},(a_1a_2)^6}\rangle$.
\end{enumerate}
\end{lem}

\begin{proof}
Note that $\Z/{2\Z}$ is generated by the image of $y$, in the sequence (\ref{TMZ2}) at the beginning of Section~\ref{BT}.
\begin{enumerate}
\item	$\M(N_{1,0})$ and $\M(N_{1,1})$ are trivial (see Theorem~3.4 in \cite{E}).
	Thus $\T(N_{1,0})$ and $\T(N_{1,1})$ are also trivial.
\item	We have the presentation
	$$\M(N_{2,0})=\langle{a_1,y}\mid{a_1^2,y^2,(a_1y)^2}\rangle$$
	(see Lemma~5 in \cite{Li1}).
	Using the Reidemeister Schreier method for this presentation, we obtain the presentation
	\begin{eqnarray*}
	\T(N_{2,0})
	&=&\langle{a_1,y^2}\mid{a_1^2,y^2,a_1^{-1}y^2a_1}\rangle\\
	&=&\langle{a_1}\mid{a_1^2}\rangle.
	\end{eqnarray*}
\item	We have the presentation
	$$\M(N_{2,1})=\langle{a_1,y}\mid{ya_1y^{-1}a_1}\rangle$$
	(see Theorem~A.7 in \cite{St1}).
	Using the Reidemeister Schreier method for this presentation, we obtain the presentation
	$$\T(N_{2,1})=\langle{a_1,y^2}\mid{a_1y^2a_1^{-1}y^{-2}}\rangle.$$
\item	We have the presentation
	$$\M(N_{3,0})=\langle{a_1,a_2,y}\mid{a_1a_2a_1a_2^{-1}a_1^{-1}a_2^{-1},(a_1a_2)^6,y^2,(a_1y)^2,(a_2y)^2}\rangle$$
	(see Theorem~3 in \cite{BC}).
	Using the Reidemeister Schreier method for this presentation, we obtain the presentation
	\begin{eqnarray*}
	\T(N_{3,0})
	&=&\langle{a_1,a_2,y^2}\mid{a_1a_2a_1a_2^{-1}a_1^{-1}a_2^{-1},(a_1a_2)^6,y^2,a_1^{-1}y^2a_1,a_2^{-1}y^2a_2}\rangle\\
	&=&\langle{a_1,a_2}\mid{a_1a_2a_1a_2^{-1}a_1^{-1}a_2^{-1},(a_1a_2)^6}\rangle.
	\end{eqnarray*}
\end{enumerate}
\end{proof}

\section{Proof of Theorem~\ref{main-thm} for $g\geq1$ and $n\leq1$}\label{main-thm-low-n}

In this section, we prove Theorem~\ref{main-thm} for $g\geq1$ and $n\leq1$, using the presentation of $\T(N_{g,n})$ given in Theorems~\ref{FPT1}, \ref{FPT0} and Lemma~\ref{FPT}.

Let $\langle{X}\mid{Y}\rangle$ be the infinitely presented group with the presentation given in Theorem~\ref{main-thm}, and $\langle{X_0}\mid{Y_0}\rangle$ the presentation for $\T(N_{g,n})$ given in Theorems~\ref{FPT1}, \ref{FPT0} or Lemma~\ref{FPT}.
Denote by $F(X_0)$ the free group freely generated by $X_0$.
Let $p:F(X_0)\to\langle{X_0}\mid{Y_0}\rangle$ be the natural projection and $\eta:F(X_0)\to\langle{X}\mid{Y}\rangle$ the homomorphism defined as $\eta(x)=x$ for any $x\in{}X_0$.
We consider a correspondence
$$\psi:\langle{X_0}\mid{Y_0}\rangle\to\langle{X}\mid{Y}\rangle$$
satisfying $\psi\circ{}p=\eta$.
Showing the following proposition, we will obtain Theorem~\ref{main-thm} for $g\geq1$ and $n\leq1$.

\begin{prop}\label{psi-isom}
For $g\geq1$ and $n\leq1$, $\psi$ is an isomorphism.
\end{prop}

Let $\varphi:\langle{X}\mid{Y}\rangle\to\langle{X_0}\mid{Y_0}\rangle=\T(N_{g,n})$ be the homomorphism defined as $\varphi(t_{c;\theta})=t_{c;\theta}$ for any $t_{c;\theta}\in{X}$.
Since $\varphi(Y)=1$ in $\T(N_{g,n})$ clearly, $\varphi$ is well-defined.
By the definitions of $\psi$ and $\varphi$, it is clear that $\varphi\circ\psi$ is the identity map if $\psi$ is a homomorphism.
Hence in order to prove Proposition~\ref{psi-isom}, what we need is to show well-definedness and surjectivity of $\psi$.
In Sections~\ref{well-def-psi} and \ref{surj-psi}, we see well-definedness and surjectivity of $\psi$ respectively.

\subsection{Well-definedness of $\psi$}\label{well-def-psi}\

In order to see well-definedness of $\psi$, it suffices to show that $\psi(r)=1$ in $\langle{X}\mid{Y}\rangle$ for any $r\in{}Y_0$.
More precisely, we check that any relation and relator of $\T(N_{g,n})$ in Theorems~\ref{FPT1}, \ref{FPT0} and Lemma~\ref{FPT} are obtained from the relations (\ref{bound})-(\ref{lantern}) of $\T(N_{g,n})$ in Theorem~\ref{main-thm} and the relation $(4)^\prime$ assigned in Remark~\ref{chain-rel}.

First we notice that the relations $(\mathrm{A7})$, $(\overline{\mathrm{A7a}})$ and $(\overline{\mathrm{A7b}})$ hold since each of the pairs $(b_0,a_1)$, $(b_1,b)$, $(\bar{b}_0,a_1^{-1})$ for $g=6$, and $(\bar{b}_1,c)$ for $g=6$, $8$ coincides as mapping classes from the definition of loops in Figure~\ref{generator}.
We may also treat the relation $(\mathrm{C1a})$ as a definition of $\varrho$.
It is easy to check that any relator of $\T(N_{g,n})$ in Lemma~\ref{FPT} is obtained from the relations (\ref{bound})-(\ref{2-chain}) in Theorem~\ref{main-thm}.
In particular, for the relator $a^2$ of $\T(N_{2,0})$, recall Remark~\ref{(2,0)}.
The relations $(\mathrm{A1})$, $(\mathrm{A2})$, $(\mathrm{A3})$, $(\mathrm{A4})$, $(\overline{\mathrm{A1}}_1)$, $(\overline{\mathrm{A1}}_2)$, $(\overline{\mathrm{A2}}_1)$, $(\overline{\mathrm{A2}}_2)$, $(\overline{\mathrm{A2}}_3)$, $(\overline{\mathrm{A3}}_1)$, $(\overline{\mathrm{A3}}_2)$, $(\overline{\mathrm{A4}})$, $(\overline{\mathrm{B3}})$, $(\overline{\mathrm{B4}}_1)$, $(\overline{\mathrm{B4}}_2)$, $(\overline{\mathrm{B7}}_1)$, $(\overline{\mathrm{B7}}_2)$, $(\mathrm{A9a})$, $(\mathrm{A9b})$, $(\overline{\mathrm{A7c}})$, $(\overline{\mathrm{A9a}})$, $(\overline{\mathrm{A9b}})$ $(\mathrm{C2})$, $(\overline{\mathrm{C2}})$, $(\overline{\mathrm{C5}}_1)$, $(\overline{\mathrm{C2}}_1)$, $(\overline{\mathrm{C2}}_2)$, $(\overline{\mathrm{C2}}_3)$ and $(\overline{\mathrm{C5}}_2)$ are obtained by repeating the relation (\ref{t-conj}) in Theorem~\ref{main-thm}.
The relations $(\mathrm{A5})$, $(\mathrm{A6})$, $(\overline{\mathrm{A5}})$, $(\overline{\mathrm{A6}})$, $(\mathrm{A8})$, $(\overline{\mathrm{A8a}})$ and $(\overline{\mathrm{A8b}})$ come from relations of mapping class groups of orientable surfaces (see Theorem~1.3 in \cite{M}), and hence these relations are obtained from the relations (\ref{t-conj}), (\ref{2-chain}) and (\ref{lanterns}) in Theorem~\ref{main-thm} by Theorem in \cite{Lu}.
The relations $(\mathrm{C3})$ and $(\mathrm{C1b})$ are obtained from the relations (\ref{bound}) in Theorem~\ref{main-thm} and $(4)^\prime$ assigned in Remark~\ref{chain-rel}.
Thus it suffices to show that the relations $(\overline{\mathrm{B1}})$, $(\overline{\mathrm{B2}}_1)$, $(\overline{\mathrm{B2}}_2)$, $(\overline{\mathrm{B6}}_1)$, $(\overline{\mathrm{B6}}_2)$, $(\overline{\mathrm{B8}}_1)$, $(\overline{\mathrm{B8}}_2)$, $(\overline{\mathrm{C1a}})$, $(\overline{\mathrm{C4a}})$, $(\overline{\mathrm{C3}})$ and $(\overline{\mathrm{C4}})$ are satisfied in $\langle{X\mid{}Y}\rangle$.

Recall the simple closed curves and the simple loop as shown in Figure~\ref{generator}.

\subsubsection{On the relation $(\overline{\mathrm{B1}})$}\label{B1}\

The relation $(\overline{\mathrm{B1}})$ can be rewritten as follows.
\begin{eqnarray*}
&&
(a_2a_3a_1a_2ea_1a_3^{-1}\underset{(\overline{\mathrm{B4}}_1)}{\underline{e)(a_2}}\underset{(\mathrm{A1})}{\underline{a_3a_1}}a_2fa_1a_3^{-1}f)=1\\
&\iff&
\underline{a_2a_3}a_1(a_2ea_1a_3^{-1}a_2ea_1a_3a_2fa_1a_3^{-1}\underset{(\overline{\mathrm{B4}}_2)}{\underline{fa_2}}\underset{(\mathrm{A1})}{\underline{a_3a_1}})a_1^{-1}\underline{a_3^{-1}a_2^{-1}}=1\\
&\iff&
a_1\{(a_2ea_1)a_3^{-1}(a_2ea_1)a_3(a_2fa_1)a_3^{-1}(a_2fa_1)a_3\}a_1^{-1}=1.
\end{eqnarray*}
Similarly, the relation $(\overline{\mathrm{B2}}_1)$ can be rewritten as follows.
\begin{eqnarray*}
y^2
&=&
a_2\underset{(\overline{\mathrm{A2}}_1)}{\underline{a_1ea_1}}a_2a_1a_2a_1\underset{(\overline{\mathrm{B4}}_2)}{\underline{a_2f}}a_1a_2\\
&=&
a_2ea_1\underset{(\overline{\mathrm{B4}}_1)}{\underline{ea_2}}a_1a_2a_1f\underset{(\mathrm{A2})}{\underline{a_2a_1a_2}}\\
&=&
a_2ea_1a_2ea_1a_2\underset{(\overline{\mathrm{A2}}_3)}{\underline{a_1fa_1}}a_2a_1\\
&=&
a_2ea_1a_2ea_1a_2fa_1\underset{(\overline{\mathrm{B4}}_2)}{\underline{fa_2}}a_1\\
&=&
(a_2ea_1)^2(a_2fa_1)^2.
\end{eqnarray*}
Let $A$, $B$, $C$, $D$ and $E$ be simple closed curves with arrows as shown in Figure~\ref{B1-loops}.
Repeating the relation (\ref{t-conj}), we have $t_A=a_1(a_2ea_1)a_3^{-1}(a_2ea_1)^{-1}a_1^{-1}$, $t_B=a_1(a_2ea_1)^2a_3(a_2ea_1)^{-2}a_1^{-1}$ and $t_C=a_1(a_2fa_1)^{-1}a_3^{-1}(a_2fa_1)a_1^{-1}$.
Assuming that $(\overline{\mathrm{B2}}_1)$ holds in $\langle{X\mid{}Y}\rangle$, we calculate
\begin{eqnarray*}
a_1\{(a_2ea_1)a_3^{-1}(a_2ea_1)a_3(a_2fa_1)a_3^{-1}(a_2fa_1)a_3\}a_1^{-1}
&=&
t_At_By^2t_C(a_1a_3a_1^{-1})\\
&\overset{(\ref{bound}),(\ref{t-conj})}{=}&
t_At_By^2t_Ca_3t_Dt_E\\
&\overset{(\ref{lanterns})}{=}&
1.
\end{eqnarray*}
Note that $B\cup\alpha_3\cup{}D\cup{}E$ bounds $\Sigma_{0,4}$.
Thus the relation $(\overline{\mathrm{B1}})$ is satisfied in $\langle{X\mid{}Y}\rangle$, by Section~\ref{B2}.

\begin{figure}[htbp]
\includegraphics{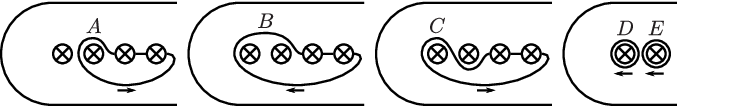}
\caption{}\label{B1-loops}
\end{figure}

\subsubsection{On the relations $(\overline{\mathrm{B2}}_1)$ and $(\overline{\mathrm{B2}}_2)$}\label{B2}\

Let $A$, $B$, $C$ and $D$ be simple closed curves with arrows as shown in Figure~\ref{B2-loops}.
Repeating the relation (\ref{t-conj}), we have $t_A=a_2a_1ea_1^{-1}a_2^{-1}$ and $t_B=a_2a_1f^{-1}a_1^{-1}a_2^{-1}$.
In addition, by the relations (\ref{bound}) and $(4)^\prime$ we have $t_C=(fa_1a_2)^4$.
Hence we calculate
\begin{eqnarray*}
y^2
&\overset{(\ref{lanterns})}{=}&
t_At_Bt_Ct_Da_1a_1^{-1}\\
&\overset{(\ref{bound})}{=}&
(a_2a_1ea_1^{-1}a_2^{-1})(a_2a_1f^{-1}a_1^{-1}a_2^{-1})(fa_1a_2)^4\\
&=&
a_2a_1ef^{-1}a_1^{-1}\underset{(\overline{\mathrm{B4}}_2)}{\underline{a_2^{-1}f}}a_1a_2fa_1\underset{(\overline{\mathrm{B4}}_2)}{\underline{a_2f}}a_1a_2fa_1a_2\\
&=&
a_2a_1e\underset{(\overline{\mathrm{A2}}_3)}{\underline{f^{-1}a_1^{-1}f}}a_2^{-1}a_1a_2\underset{(\overline{\mathrm{A2}}_3)}{\underline{fa_1f}}a_2a_1a_2fa_1a_2\\
&=&
a_2a_1ea_1\underset{(\mathrm{A2}),(\overline{\mathrm{B4}}_2)}{\underline{f^{-1}a_1^{-1}a_2^{-1}a_1a_2a_1f}}a_1a_2a_1a_2fa_1a_2\\
&=&
a_2a_1ea_1a_2a_1a_2a_1a_2fa_1a_2.
\end{eqnarray*}
Note that $C\cup{}D\cup\alpha_1$ bounds $\Sigma_{0,4}$.
Thus the relation $(\overline{\mathrm{B2}}_1)$ is satisfied in $\langle{X\mid{}Y}\rangle$.

\begin{figure}[htbp]
\includegraphics{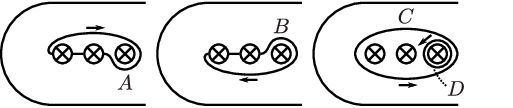}
\caption{}\label{B2-loops}
\end{figure}

Moreover, similar to the transformation of the relation $(\overline{\mathrm{B2}}_1)$ of Section~\ref{B1}, the relation $(\overline{\mathrm{B2}}_2)$ can be rewritten as
$$(a_2ea_1)^2(a_2fa_1)^4(a_2ea_1)^2=1.$$
By the relations (\ref{bound}) and $(4)^\prime$, we have $t_C=(a_2fa_1)^4$ and $t_C^{-1}=(a_2ea_1)^4$.
Hence we calculate
\begin{eqnarray*}
(a_2ea_1)^2(a_2fa_1)^4(a_2ea_1)^2
&=&
(a_2ea_1)^2t_C(a_2ea_1)^2\\
&\overset{(\ref{t-conj})}{=}&
(a_2ea_1)^4t_C\\
&=&
t_C^{-1}t_C\\
&=&
1.
\end{eqnarray*}
Thus the relation $(\overline{\mathrm{B2}}_2)$ is satisfied in $\langle{X\mid{}Y}\rangle$.

\subsubsection{On the relations $(\overline{\mathrm{B6}}_1)$ and $(\overline{\mathrm{B6}}_2)$}\

Let $A$, $B$, $C$, $D$ and $E$ be simple closed curves with arrows as shown in Figure~\ref{B6-loops}.
Repeating the relation (\ref{t-conj}), we have $t_A=a_1a_2a_3f^{-1}a_3^{-1}a_2^{-1}a_1^{-1}$ and $t_B=a_2^{-1}a_3^{-1}e^{-1}a_3a_2$.
Then we calculate
\begin{eqnarray*}
bcy^2
&\overset{(\ref{lanterns})}{=}&
t_Ca_1a_1^{-1}a_3\\
&=&
t_Ca_3,\\
t_At_By^2
&\overset{(\ref{lanterns})}{=}&
t_Ct_Dt_Ea_3\\
&\overset{(\ref{bound})}{=}&
t_Ca_3,
\end{eqnarray*}
and hence $bc=t_At_B$.
Note that $C\cup\alpha_1\cup\alpha_3$ and $C\cup{}D\cup{}E\cup\alpha_3$ bound $\Sigma_{0,4}$.
Thus the relation $(\overline{\mathrm{B6}}_1)$ is satisfied in $\langle{X\mid{}Y}\rangle$.

\begin{figure}[htbp]
\includegraphics{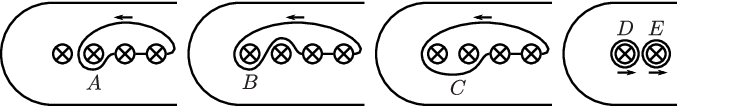}
\caption{}\label{B6-loops}
\end{figure}

In addition, repeating the relation (\ref{t-conj}), we have $y^2by^{-2}=t_{y(\gamma)}$, $a_1^{-1}e^{-1}a_3a_2a_3^{-1}ea_1=t_{y(A)}$ and $ea_3^{-1}y^2a_2y^{-2}a_3e^{-1}=t_{y(B)}$.
Then we calculate
\begin{eqnarray*}
t_{y(\beta)}t_{y(\gamma)}y^2
&\overset{(\ref{lanterns})}{=}&
t_Ca_1a_1^{-1}a_3\\
&=&
t_Ca_3,\\
t_{y(A)}t_{y(B)}y^2
&\overset{(\ref{lanterns})}{=}&
t_Ct_{y(D)}t_{y(E)}a_3\\
&\overset{(\ref{bound})}{=}&
t_Ca_3,
\end{eqnarray*}
and hence $t_{y(\beta)}t_{y(\gamma)}=t_{y(A)}t_{y(B)}$.
Note that $C\cup{}y(D)\cup{}y(E)\cup\alpha_3=y(C\cup{}D\cup{}E\cup\alpha_3)$ bounds $\Sigma_{0,4}$.
Thus, since $c=t_{y(\beta)}$, the relation $(\overline{\mathrm{B6}}_2)$ is satisfied in $\langle{X\mid{}Y}\rangle$.

\subsubsection{On the relations $(\overline{\mathrm{B8}}_1)$ and $(\overline{\mathrm{B8}}_2)$}\

Let $A$, $B$, $C$ and $D$ be simple closed curves with arrows as shown in Figure~\ref{B8-loops}.
Repeating the relation (\ref{t-conj}), we have $t_A=(a_1ea_3^{-1}a_4^{-1})c(a_4a_3e^{-1}a_1^{-1})$, $t_B=(a_1^{-1}a_2^{-1}a_3^{-1}a_4^{-1})b^{-1}(a_4a_3a_2a_1)$ and $t_C=(a_3^{-1}a_2^{-1}e^{-1}a_3)a_4(a_3^{-1}ea_2a_3)$.
By the relation (\ref{lanterns}) we see
$$t_At_Da_2a_4=t_Ce^{-1}t_B^{-1},$$
and by the relations (\ref{bound}) and (\ref{t-conj}) we see
$$t_At_B=a_4^{-1}t_Ca_2^{-1}e^{-1}.$$
Note that $A\cup{}D\cup\alpha_2\cup\alpha_4$ bounds $\Sigma_{0,4}$.
Thus the relation $(\overline{\mathrm{B8}}_1)$ is satisfied in $\langle{X\mid{}Y}\rangle$.

\begin{figure}[htbp]
\includegraphics{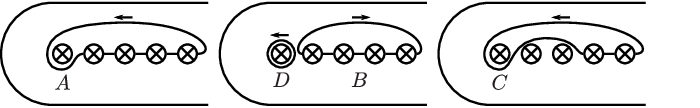}
\caption{}\label{B8-loops}
\end{figure}

In addition, repeating the relation (\ref{t-conj}), we have $(a_1^{-1}a_2^{-1}a_3^{-1}a_4^{-1})b(a_4a_3a_2a_1)=t_{y^{-1}(A)}$, $(a_1fa_3^{-1}a_4^{-1})y^{-2}c^{-1}y^2(a_4a_3f^{-1}a_1^{-1})=t_{y^{-1}(B)}$ and $(a_3^{-1}fa_2a_3)a_4(a_3^{-1}a_2^{-1}f^{-1}a_3)=t_{y^{-1}(C)}$.
By the relation (\ref{lanterns}) we see
$$t_{y^{-1}(A)}t_{y^{-1}(D)}f^{-1}a_4=t_{y^{-1}(C)}a_2t_{y^{-1}(B)}^{-1},$$
and by the relations (\ref{bound}) and (\ref{t-conj}) we see
$$t_{y^{-1}(A)}t_{y^{-1}(B)}=a_4^{-1}t_{y^{-1}(C)}fa_2.$$
Note that $y^{-1}(A)\cup{}y^{-1}(D)\cup\zeta\cup\alpha_4=y^{-1}(A\cup{}D\cup\alpha_2\cup\alpha_4)$ bounds $\Sigma_{0,4}$.
Thus the relation $(\overline{\mathrm{B8}}_2)$ is satisfied in $\langle{X\mid{}Y}\rangle$.

\subsubsection{On the relation $(\overline{\mathrm{C1a}})$}\label{C1a}\

We calculate
\begin{eqnarray*}
\varrho
&\overset{(\mathrm{C1a})}{=}&
(a_1a_2\cdots{}a_{g-1})^g\\
&=&
a_1a_2\cdots{}a_{g-1}\underset{(\mathrm{A1}),(\mathrm{A2})}{\underline{(a_1a_2\cdots{}a_{g-1})^{g-2}a_1}}a_2\cdots{}a_{g-1}\\
&=&
a_1a_2\cdots{}a_{g-1}a_{g-1}(a_1a_2\cdots{}a_{g-1})^{g-2}a_2\cdots{}a_{g-1}\\
&=&
a_1a_2\cdots{}a_{g-1}a_{g-1}\underset{(\mathrm{A1}),(\mathrm{A2})}{\underline{(a_1a_2\cdots{}a_{g-1})^{g-3}a_1}}(a_2\cdots{}a_{g-1})^2\\
&=&
a_1a_2\cdots{}a_{g-1}a_{g-1}a_{g-2}(a_1a_2\cdots{}a_{g-1})^{g-3}(a_2\cdots{}a_{g-1})^2\\
&\vdots&\\
&=&
a_1a_2a_3\cdots{}a_{g-1}a_{g-1}\cdots{}a_3a_2a_1(a_2a_3\cdots{}a_{g-1})^{g-1}\\
&\overset{(\mathrm{A1}),(\mathrm{A2})}{=}&
(a_2a_3\cdots{}a_{g-1})^{g-1}a_1a_2a_3\cdots{}a_{g-1}a_{g-1}\cdots{}a_3a_2a_1.
\end{eqnarray*}
Let $\varrho^\prime=(a_1^{-1}e^{-1}a_3\cdots{}a_{g-1})^g$.
Similarly, we calculate
\begin{eqnarray*}
\varrho^\prime
&=&
(a_1^{-1}e^{-1}a_3\cdots{}a_{g-1})^g\\
&=&
a_1^{-1}e^{-1}a_3\cdots{}a_{g-1}\underset{(\mathrm{A1}),(\mathrm{A2}),(\overline{\mathrm{A1}}_1),(\overline{\mathrm{A2}}_1)(\overline{\mathrm{A2}}_2)}{\underline{(a_1^{-1}e^{-1}a_3\cdots{}a_{g-1})^{g-2}a_1^{-1}}}e^{-1}a_3\cdots{}a_{g-1}\\
&=&
a_1^{-1}e^{-1}a_3\cdots{}a_{g-1}a_{g-1}(a_1^{-1}e^{-1}a_3\cdots{}a_{g-1})^{g-2}e^{-1}a_3\cdots{}a_{g-1}\\
&=&
a_1^{-1}e^{-1}a_3\cdots{}a_{g-1}a_{g-1}\underset{(\mathrm{A1}),(\mathrm{A2}),(\overline{\mathrm{A1}}_1),(\overline{\mathrm{A2}}_1)(\overline{\mathrm{A2}}_2)}{\underline{(a_1^{-1}e^{-1}a_3\cdots{}a_{g-1})^{g-3}a_1^{-1}}}(e^{-1}a_3\cdots{}a_{g-1})^2\\
&=&
a_1^{-1}e^{-1}a_3\cdots{}a_{g-1}a_{g-1}a_{g-2}(a_1^{-1}e^{-1}a_3\cdots{}a_{g-1})^{g-3}(e^{-1}a_3\cdots{}a_{g-1})^2\\
&\vdots&\\
&=&
a_1^{-1}e^{-1}a_3\cdots{}a_{g-1}a_{g-1}\cdots{}a_3e^{-1}a_1^{-1}(e^{-1}a_3\cdots{}a_{g-1})^{g-1}\\
&\underset{(\overline{\mathrm{A1}}_1),(\overline{\mathrm{A2}}_1)(\overline{\mathrm{A2}}_2)}{\overset{(\mathrm{A1}),(\mathrm{A2})}{=}}&
(a_2a_3\cdots{}a_{g-1})^{g-1}a_1a_2a_3\cdots{}a_{g-1}a_{g-1}\cdots{}a_3a_2a_1.
\end{eqnarray*}
We show that $\varrho^\prime=y^2\varrho$.
Since a regular neighborhood of $\alpha_1\cup\varepsilon\cup\alpha_3\cup\cdots\cup\alpha_{g-1}$ is diffeomorphic to $\Sigma_{\frac{g-1}{2},1}$, by the relations (\ref{bound}) and $(4)^\prime$, we have $(\varrho^\prime)^2=1$.
We calculate
\begin{eqnarray*}
\varrho^\prime\varrho^{-1}
&=&
(\varrho^\prime)^{-1}\varrho^{-1}\\
&=&
a_1ea_3^{-1}\cdots{}a_{g-1}^{-1}a_{g-1}^{-1}\cdots{}a_3^{-1}ea_1\underset{(\mathrm{C2}),(\overline{\mathrm{C2}})}{\underline{(e^{-1}a_3\cdots{}a_{g-1})^{1-g}\varrho^{-1}}}\\
&=&
a_1ea_3^{-1}\cdots{}a_{g-1}^{-1}a_{g-1}^{-1}\cdots{}a_3^{-1}ea_1\underline{\varrho^{-1}}(f^{-1}a_3\cdots{}a_{g-1})^{1-g}\\
&=&
a_1ea_3^{-1}\cdots{}a_{g-1}^{-1}a_{g-1}^{-1}\cdots{}a_3^{-1}\underline{ea_1}\\
&&
\underline{a_1^{-1}a_2^{-1}}a_3^{-1}\cdots{}a_{g-1}^{-1}a_{g-1}^{-1}\cdots{}a_3^{-1}a_2^{-1}a_1^{-1}(a_2a_3\cdots{}a_{g-1})^{1-g}\\
&&
(f^{-1}a_3\cdots{}a_{g-1})^{1-g}\\
&\overset{(\overline{\mathrm{B4}}_1)}{=}&
a_1\{ea_2(a_2^{-1}a_3^{-1}\cdots{}a_{g-1}^{-1}a_{g-1}^{-1}\cdots{}a_3^{-1}a_2^{-1})\}^2a_1^{-1}\\
&&(a_2a_3\cdots{}a_{g-1})^{1-g}(f^{-1}a_3\cdots{}a_{g-1})^{1-g}.
\end{eqnarray*}
It suffices to show the equality
\begin{align*}
&y^2(f^{-1}a_3\cdots{}a_{g-1})^{g-1}(a_2a_3\cdots{}a_{g-1})^{g-1}\\&=a_1\{ea_2(a_2^{-1}a_3^{-1}\cdots{}a_{g-1}^{-1}a_{g-1}^{-1}\cdots{}a_3^{-1}a_2^{-1})\}^2a_1^{-1}.
\end{align*}

Let $\gamma_1$, $\gamma_2$ and $\gamma_3$ be oriented simple loops based at $\ast$ which is a point obtained by the blowdown with respect to the first crosscap in $N_{g,0}$ when $g$ is odd, as shown in Figure~\ref{C1a-loops}~(a), and $A$ and $B$ simple closed curves with arrows as shown in Figure~\ref{C1a-loops}~(b).
By the relations (\ref{bound}), $(4)^\prime$ and Remark~\ref{CP1}, we have $y^2=\CP(\alpha^2)$, $(a_2a_3\cdots{}a_{g-1})^{g-1}=\CP(\gamma_1)$, $(f^{-1}a_3\cdots{}a_{g-1})^{g-1}=\CP(\gamma_2)$ and $t_At_B=\CP(\gamma_3)$.
Hence by Remark~\ref{CP2}, we calculate
\begin{eqnarray*}
y^2(f^{-1}a_3\cdots{}a_{g-1})^{g-1}(a_2a_3\cdots{}a_{g-1})^{g-1}
&=&
\CP(\alpha^2)\CP(\gamma_2)\CP(\gamma_1)\\
&=&
\CP(\alpha^2)\CP(\gamma_1\gamma_2)\\
&=&
\CP(\gamma_1\gamma_2\alpha^2)\\
&=&
\CP(\gamma_3)\\
&=&
t_At_B\\
&\overset{(\ref{bound})}{=}&
t_A.
\end{eqnarray*}
On the other hand, we see
\begin{align*}
&
a_1\{ea_2(a_2^{-1}a_3^{-1}\cdots{}a_{g-1}^{-1}a_{g-1}^{-1}\cdots{}a_3^{-1}a_2^{-1})\}^2a_1^{-1}\\
=&
a_1(ea_2)a_1^{-1}\cdot{}a_1(a_2a_3\cdots{}a_{g-1}a_{g-1}\cdots{}a_3a_2)^{-1}ea_2(a_2a_3\cdots{}a_{g-1}a_{g-1}\cdots{}a_3a_2)a_1^{-1}\\
&
a_1(a_2a_3\cdots{}a_{g-1}a_{g-1}\cdots{}a_3a_2)^{-2}a_1^{-1}.
\end{align*}
In addition, by the calculation similar to the beginning of Section~\ref{C1a}, we see
\begin{eqnarray*}
(a_2a_3\cdots{}a_{g-1})^{2(g-1)}
&=&
\left\{(a_2a_3\cdots{}a_{g-1})^{g-1}\right\}^2\\
&\overset{(\mathrm{A1}),(\mathrm{A2})}{=}&
\left\{(a_3\cdots{}a_{g-1})^{g-2}(a_2a_3\cdots{}a_{g-1}a_{g-1}\cdots{}a_3a_2)\right\}^2\\
&\overset{(\mathrm{A1}),(\mathrm{A2})}{=}&
(a_3\cdots{}a_{g-1})^{2(g-2)}(a_2a_3\cdots{}a_{g-1}a_{g-1}\cdots{}a_3a_2)^2,
\end{eqnarray*}
and hence
\begin{eqnarray*}
&&
a_1(a_2a_3\cdots{}a_{g-1}a_{g-1}\cdots{}a_3a_2)^{-2}a_1^{-1}\\
&=&
a_1(a_2a_3\cdots{}a_{g-1})^{-2(g-1)}(a_3\cdots{}a_{g-1})^{2(g-2)}a_1^{-1}\\
&\overset{(\mathrm{A1})}{=}&
a_1(a_2a_3\cdots{}a_{g-1})^{-2(g-1)}a_1^{-1}(a_3\cdots{}a_{g-1})^{2(g-2)}.
\end{eqnarray*}
Let $\gamma_4$, $\gamma_5$, $\gamma_6$ and $\gamma_7$ be oriented simple loops based at $\ast$ which is a point obtained by the blowdown with respect to the second crosscap, as shown in Figure~\ref{C1a-loops}~(a), $\gamma_{5,1}$, $\gamma_{5,2}$ and $\gamma_{5,3}$ oriented simple loops based at $\ast$ which is a point obtained by the blowdown with respect to the first crosscap, as shown in Figure~\ref{C1a-loops}~(b), and $C$ a simple closed curve with an arrow as shown in Figure~\ref{C1a-loops}~(c).
By the relations (\ref{t-conj}), $(4)^\prime$ and Remark~\ref{CP1}, we have $a_1(a_2a_3\cdots{}a_{g-1})^{-(g-1)}a_1^{-1}=\CP(\gamma_4)$,
\begin{eqnarray*}
&&a_1(a_2a_3\cdots{}a_{g-1}a_{g-1}\cdots{}a_3a_2)^{-1}ea_2(a_2a_3\cdots{}a_{g-1}a_{g-1}\cdots{}a_3a_2)a_1^{-1}\\
&=&a_1(a_2a_3\cdots{}a_{g-1}a_{g-1}\cdots{}a_3a_2)^{-1}\CP(\gamma_{5,1})(a_2a_3\cdots{}a_{g-1}a_{g-1}\cdots{}a_3a_2)a_1^{-1}\\
&=&a_1(a_{g-1}\cdots{}a_3a_2)^{-1}\CP(\gamma_{5,2})(a_{g-1}\cdots{}a_3a_2)a_1^{-1}\\
&=&a_1\CP(\gamma_{5,3})a_1^{-1}\\
&=&\CP(\gamma_5),
\end{eqnarray*}
$a_1(ea_2)a_1^{-1}=\CP(\gamma_6)$, $t_At_C^{-1}=\CP(\gamma_7)$ and $(a_3\cdots{}a_{g-1})^{2(g-2)}=t_C$.
Hence by Remark~\ref{CP2}, we calculate
\begin{eqnarray*}
a_1\{ea_2(a_2^{-1}a_3^{-1}\cdots{}a_{g-1}^{-1}a_{g-1}^{-1}\cdots{}a_3^{-1}a_2^{-1})\}^2a_1^{-1}
&=&
\CP(\gamma_6)\CP(\gamma_5)\CP(\gamma_4)^2t_C\\
&=&
\CP(\gamma_6)\CP(\gamma_4\gamma_5)\CP(\gamma_4)t_C\\
&=&
\CP(\gamma_4\gamma_5\gamma_6)\CP(\gamma_4)t_C\\
&=&
\CP(\gamma_4^2\gamma_5\gamma_6)t_C\\
&=&
\CP(\gamma_7)t_C\\
&=&
t_A.
\end{eqnarray*}
Thus the relation $(\overline{\mathrm{C1a}})$ is satisfied in $\langle{X\mid{}Y}\rangle$.

\begin{figure}[htbp]
\subfigure[]{\includegraphics{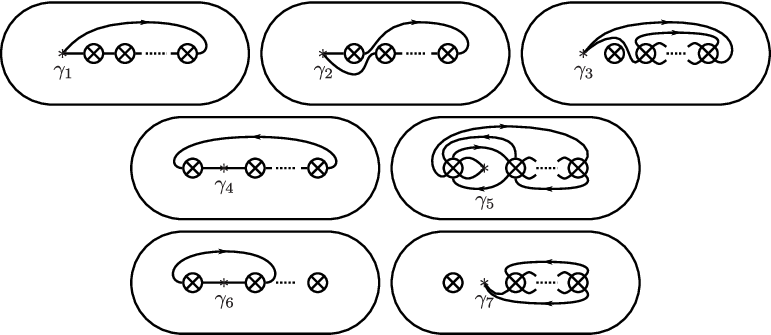}}
\subfigure[]{\includegraphics{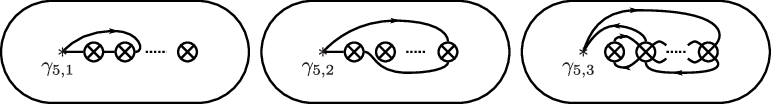}}
\subfigure[]{\includegraphics{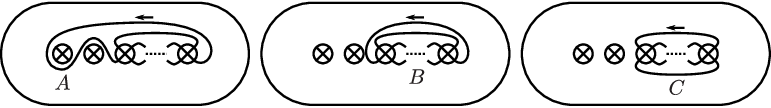}}
\caption{}\label{C1a-loops}
\end{figure}

\subsubsection{On the relation $(\overline{\mathrm{C4a}})$}\label{C4a}\

For $\displaystyle1\leq{i}\leq{\frac{g-1}{2}}$, let $\gamma_i$ be an oriented simple loop based at $\ast$ which is a point obtained by the blowdown with respect to the first crosscap, as shown in Figure~\ref{C4a-loops}, and $\Phi=e^{-1}a_3\cdots{}a_{g-1}$.

We now prove the following lemma.

\begin{lem}\label{Phi-alpha_i}
In $\langle{X\mid{}Y}\rangle$, we have $\CP(\gamma_1)\Phi^2\CP(\gamma_i)=\CP(\gamma_{i+1})\Phi^2$ for $\displaystyle1\leq{i}\leq{\frac{g-3}{2}}$.
\end{lem}

\begin{proof}
Let $\delta_i$ be an oriented simple loop based at $\ast$ which is a point obtained by the blowdown with respect to the first crosscap, as shown in Figure~\ref{C4a-loops}.
By Remark~\ref{CP1}, repeating the relation~(\ref{t-conj}), we have $\Phi^2\CP(\gamma_i)\Phi^{-2}=\CP(\delta_i)$.
In addition, by Remarks~\ref{CP1} and \ref{CP2}, we have $\CP(\gamma_1)\CP(\delta_i)=\CP(\gamma_{i+1})$, and hence the claim holds.
\end{proof}

\begin{figure}[htbp]
\includegraphics{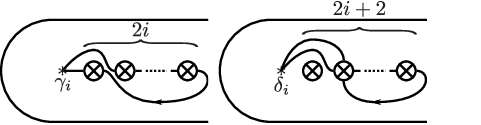}
\caption{}\label{C4a-loops}
\end{figure}

By the relation $(4)^\prime$ and Remark~\ref{CP1}, we have $a_2e=\CP(\gamma_1)$ and $\CP(\gamma_\frac{g-1}{2})=\Phi^{1-g}$.
Hence by Lemma~\ref{Phi-alpha_i}, we calculate
\begin{eqnarray*}
(a_2a_3\cdots{}a_{g-1}e^{-1}a_3\cdots{}a_{g-1})^{\frac{g-1}{2}}
&=&
(a_2e(e^{-1}a_3\cdots{}a_{g-1})^2)^{\frac{g-1}{2}}\\
&=&
(\CP(\gamma_1)\Phi^2)^{\frac{g-1}{2}}\\
&=&
(\CP(\gamma_1)\Phi^2)^{\frac{g-5}{2}}\CP(\gamma_2)\Phi^4\\
&=&
(\CP(\gamma_1)\Phi^2)^{\frac{g-7}{2}}\CP(\gamma_3)\Phi^6\\
&\vdots&\\
&=&
\CP(\gamma_{\frac{g-1}{2}})\Phi^{g-1}\\
&=&
1.
\end{eqnarray*}
Thus the relation $(\overline{\mathrm{C4a}})$ is satisfied in $\langle{X\mid{}Y}\rangle$.

\subsubsection{On the relation $(\overline{\mathrm{C3}})$}\label{(C3)}\

In $\M(N_{g,0})$, $\bar{\varrho}$ is defined as $\bar{\varrho}=y\varrho$ and we have the relation $\varrho=(y^{-1}a_2a_3\cdots{}a_{g-1}ya_2a_3\cdots{}a_{g-1})^{\frac{g-2}{2}}y^{-1}a_2a_3\cdots{}a_{g-1}$ (see Theorem~2.2 in \cite{St4}).
In addition, note that Lemma~\ref{Phi-alpha_i} holds even if $g$ is even.
Hence we calculate
\begin{eqnarray*}
\bar{\varrho}
&=&
y(y^{-1}a_2a_3\cdots{}a_{g-1}ya_2a_3\cdots{}a_{g-1})^{\frac{g-2}{2}}y^{-1}a_2a_3\cdots{}a_{g-1}\\
&=&
(a_2a_3\cdots{}a_{g-1}ya_2a_3\cdots{}a_{g-1}y^{-1})^{\frac{g-2}{2}}a_2a_3\cdots{}a_{g-1}\\
&=&
(a_2a_3\cdots{}a_{g-1}e^{-1}a_3\cdots{}a_{g-1})^{\frac{g-2}{2}}a_2a_3\cdots{}a_{g-1}\\
&=&
(\CP(\gamma_1)\Phi^2)^{\frac{g-2}{2}}\CP(\gamma_1)\Phi\\
&=&
(\CP(\gamma_1)\Phi^2)^{\frac{g-4}{2}}\CP(\gamma_2)\Phi^3\\
&=&
(\CP(\gamma_1)\Phi^2)^{\frac{g-6}{2}}\CP(\gamma_3)\Phi^5\\
&\vdots&\\
&=&
\CP(\gamma_1)\Phi^2\CP(\gamma_{\frac{g-2}{2}})\Phi^{g-3}\\
&=&
\CP(\gamma_1)(\Phi^2\CP(\gamma_{\frac{g-2}{2}})\Phi^{-2})\Phi^{g-1},
\end{eqnarray*}
where $\gamma_i$ is defined in Section~\ref{C4a}.
Let $\delta_1$ and $\delta_2$ be oriented simple loops based at $\ast$ which is a point obtained by the blowdown with respect to the first crosscap, as shown in Figure~\ref{C3-loops}~(a).
Repeating the relation~(\ref{t-conj}), we have $\Phi^2\CP(\gamma_{\frac{g-2}{2}})\Phi^{-2}=\CP(\delta_1)$.
In addition, by Remark~\ref{CP2}, we have $\CP(\gamma_1)\CP(\delta_1)=\CP(\delta_2)$.
Let $\delta_3$ and $\delta_4$ be oriented simple loops based at $\ast$ which is a point obtained by the blowdown with respect to the first crosscap, as shown in Figure~\ref{C3-loops}~(a).
Repeating the relation~(\ref{t-conj}), we have $\Phi^{g-1}\CP(\delta_2)\Phi^{1-g}=\CP(\delta_3)$.
In addition, by Remark~\ref{CP2}, we have $\CP(\delta_2)\CP(\delta_3)=\CP(\delta_4)$.
Let $A$ and $B$ be simple closed curves with arrows as shown in Figure~\ref{C3-loops}~(b).
By the relation $(4)^\prime$ and Remark~\ref{CP1}, we have $\Phi^{2g-2}=t_A$ and $\CP(\delta_4)=t_Bt_A^{-1}$.
Hence we calculate
\begin{eqnarray*}
\bar{\varrho}^2
&=&
(\CP(\delta_2)\Phi^{g-1})^2.\\
&=&
\CP(\delta_2)\CP(\delta_3)\Phi^{2g-2}\\
&=&
\CP(\delta_4)t_A\\
&=&
t_B\\
&\overset{(\ref{bound})}{=}&
1.
\end{eqnarray*}
Thus the relation $(\overline{\mathrm{C3}})$ is satisfied in $\langle{X\mid{}Y}\rangle$.

\begin{figure}[htbp]
\subfigure[]{\includegraphics{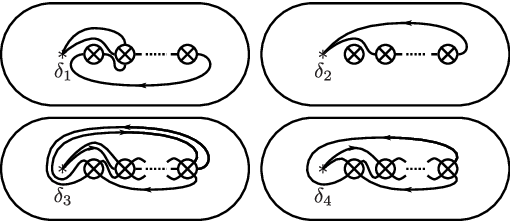}}
\subfigure[]{\includegraphics{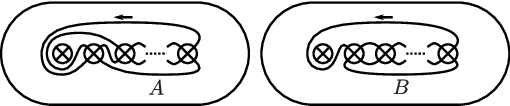}}
\caption{}\label{C3-loops}
\end{figure}

\subsubsection{On the relation $(\overline{\mathrm{C4}})$}\ 

Note that the relation $a_2\bar{\varrho}=\bar{\varrho}e^{-1}$ is obtained from the relations $(\overline{\mathrm{C2}}_3)$ and $(\overline{\mathrm{C3}})$.
We call this relation $(\overline{\mathrm{C2}}_4)$.
By the equality $\bar{\varrho}=(a_2a_3\cdots{}a_{g-1}e^{-1}a_3\cdots{}a_{g-1})^{\frac{g-2}{2}}a_2a_3\cdots{}a_{g-1}$ which appeared in Section~\ref{(C3)}, we calculate
\begin{eqnarray*}
(\bar{\varrho}a_2a_3\cdots{}a_{g-1})^{g-1}
&=&
\bar{\varrho}\underset{(\overline{\mathrm{C2}}_2),(\overline{\mathrm{C2}}_4)}{\underline{a_2a_3\cdots{}a_{g-1}\bar{\varrho}}}a_2a_3\cdots{}a_{g-1}(\bar{\varrho}a_2a_3\cdots{}a_{g-1})^{g-3}\\
&=&
\bar{\varrho}^2\underset{(\overline{\mathrm{C2}}_2),(\overline{\mathrm{C2}}_3),(\overline{\mathrm{C2}}_4)}{\underline{e^{-1}a_3\cdots{}a_{g-1}a_2a_3\cdots{}a_{g-1}\bar{\varrho}}}a_2a_3\cdots{}a_{g-1}(\bar{\varrho}a_2a_3\cdots{}a_{g-1})^{g-4}\\
&=&
\bar{\varrho}^3a_2a_3\cdots{}a_{g-1}e^{-1}a_3\cdots{}a_{g-1}a_2a_3\cdots{}a_{g-1}(\bar{\varrho}a_2a_3\cdots{}a_{g-1})^{g-4}\\
&\vdots&\\
&=&
\bar{\varrho}^{g-1}(a_2a_3\cdots{}a_{g-1}e^{-1}a_3\cdots{}a_{g-1})^{\frac{g-2}{2}}a_2a_3\cdots{}a_{g-1}\\
&=&
\bar{\varrho}^g\\
&\overset{(\overline{\mathrm{C3}})}{=}&
1.
\end{eqnarray*}
Thus the relation $(\overline{\mathrm{C4}})$ is satisfied in $\langle{X\mid{}Y}\rangle$.

Therefore well-definedness of $\psi$ follows.

\subsection{Surjectivity of $\psi$}\label{surj-psi}\

For any simple closed curve $d$ of $N_{g,n}$, we denote the complement of the interior of a regular neighborhood of $d$ by $N_{g,n}\setminus{d}$.
If $N_{g,n}\setminus{d}$ and $N_{g,d}\setminus{d^\prime}$ are diffeomorphic, there exists $x\in\M(N_{g,n})$ with $x(d^\prime)=d$ for $n\leq1$.
In particular, if $d$ and $d^\prime$ are two sided, we have $xt_{d^\prime;\theta^\prime}x^{-1}=t_{d;\theta}$ for some orientation $\theta$ and $\theta^\prime$.

For any two sided simple closed curve $d$ of $N_{g,n}$ where $n\leq1$, $N_{g,n}\setminus{d}$ is diffeomorphic to either one of
\begin{itemize}
\item	$N_{g-2,n+2}$, where $g\geq3$,
\item	$\Sigma_{\frac{g-2}{2},n+2}$, where $g$ is even,
\item	$\Sigma_{h,1}\sqcup{}N_{g-2h,n+1}$, where $\displaystyle0\leq{h}<\frac{g}{2}$,
\item	$N_{i,1}\sqcup{}N_{g-i,n+1}$, where $1\leq{i}\leq{}g-1$ and $g\geq2$,
\item	$N_{i,1}\sqcup\Sigma_{\frac{g-i}{2},n+1}$, where $1\leq{i}\leq{}g$ and $g-i$ is even.
\end{itemize}
For each case, we would like to find a word $w$ on $X_0$ such that $\psi(w)=t_{d;\theta}$, where $X_0$ is defined at the beginning of Section~\ref{main-thm-low-n}.

\subsubsection{The case where $N_{g,n}\setminus{d}$ is diffeomorphic to $N_{g-2,n+2}$ or $\Sigma_{\frac{g-2}{2},n+2}$}\

Since $N_{g,n}\setminus{d}$ is diffeomorphic to either $N_{g,n}\setminus{\alpha_1}$ or $N_{g,n}\setminus{\beta_\frac{g-2}{2}}$, there exists $x\in\M(N_{g,n})$ such that $xt_{d^\prime;\theta^\prime}x^{-1}=t_{d;\theta}$, where $t_{d^\prime;\theta^\prime}=a_1$ or $b_{\frac{g-2}{2}}$ respectively.
If $x\in\T(N_{g,n})$, there exists a word $x=x_1x_2\cdots{}x_s$ on $X_0$.
Then we obtain $\psi((x_1x_2\cdots{}x_s)t_{d^\prime;\theta^\prime}(x_1x_2\cdots{}x_s)^{-1})=t_{d;\theta}$, repeating the relation (\ref{t-conj}).
If $x\notin\T(N_{g,n})$, since $xy^{-1}\in\T(N_{g,n})$ by the sequence (\ref{TMZ2}) in Section~\ref{BT}, there exists a word $xy^{-1}=x_1x_2\cdots{}x_s$ on $X_0$.
Since $yt_{d^\prime;\theta^\prime}y^{-1}=a_1^{-1}$ or $\bar{b}_{\frac{g-2}{2}}$, we obtain $\psi((x_1x_2\cdots{}x_s)(yt_{d^\prime;\theta^\prime}y^{-1})(x_1x_2\cdots{}x_s)^{-1})=t_{d;\theta}$, repeating the relation (\ref{t-conj}).

\subsubsection{The case where $N_{g,n}\setminus{d}$ is diffeomorphic to $\Sigma_{h,1}\sqcup{}N_{g-2h,n+1}$}\

When $h=0$, since $t_{d;\theta}=1$ by the relation (\ref{bound}), we obtain $\psi(1)=t_{d;\theta}$.
When $h\geq1$, there exists $x\in\M(N_{g,n})$ such that $xt_{d^\prime;\theta^\prime}x^{-1}=t_{d;\theta}$, where $d^\prime$ is the boundary curve of a regular neighborhood of $\alpha_1\cup\alpha_2\cup\cdots\cup\alpha_{2h}$ which is diffeomorphic to $\Sigma_{h,1}$, as shown in Figure~\ref{3.2.2}.
Note that we have $t_{d^\prime;\theta^\prime}^\epsilon=(a_1a_2\cdots{}a_{2h})^{4h+2}$ for some $\epsilon=\pm1$, by the relation $(4)^\prime$.
If $x\in\T(N_{g,n})$, there exists a word $x=x_1x_2\cdots{}x_s$ on $X_0$.
Then we obtain $\psi((x_1x_2\cdots{}x_s)(a_1a_2\cdots{}a_{2h})^{\epsilon(4h+2)}(x_1x_2\cdots{}x_s)^{-1})=t_{d;\theta}$, repeating the relations (\ref{t-conj}) and $(4)^\prime$.
If $x\notin\T(N_{g,n})$, there exists a word $xy^{-1}=x_1x_2\cdots{}x_s$ on $X_0$.
Since $ya_1y^{-1}=a_1^{-1}$, $ya_2y^{-1}=e^{-1}$ and $ya_iy^{-1}=a_i$ for $i\geq3$, we obtain $\psi((x_1x_2\cdots{}x_s)(a_1^{-1}e^{-1}a_3\cdots{}a_{2h})^{\epsilon(4h+2)}(x_1x_2\cdots{}x_s)^{-1})=t_{d;\theta}$, repeating the relations (\ref{t-conj}) and $(4)^\prime$.

\begin{figure}[htbp]
\includegraphics{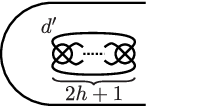}
\caption{}\label{3.2.2}
\end{figure}

\subsubsection{The case where $N_{g,n}\setminus{d}$ is diffeomorphic to $N_{i,1}\sqcup{}N_{g-i,n+1}$}\label{N_{i,1}-1}\

When $i=1$, since $t_{d;\theta}=1$ by the relation (\ref{bound}), we obtain $\psi(1)=t_{d;\theta}$.
When $i=2$, there exists $x\in\M(N_{g,n})$ such that $xy^2x^{-1}=xt_\delta{}x^{-1}=t_{d;\theta}$, where $\delta$ appeared in Figure~\ref{generator}.
If $x\in\T(N_{g,n})$, there exists a word $x=x_1x_2\cdots{}x_s$ on $X_0$.
Then we obtain $\psi((x_1x_2\cdots{}x_s)y^2(x_1x_2\cdots{}x_s)^{-1})=t_{d;\theta}$, repeating the relation (\ref{t-conj}).
If $x\notin\T(N_{g,n})$, there exists a word $xy^{-1}=x_1x_2\cdots{}x_s$ on $X_0$.
Since $y\cdot{}y^2\cdot{}y^{-1}=y^2$, we have $\psi((x_1x_2\cdots{}x_s)y^2(x_1x_2\cdots{}x_s)^{-1})=t_{d;\theta}$, repeating the relation (\ref{t-conj}).
When $i\geq3$, we take simple closed curves $d_1,\dots,d_6$ as shown in Figure~\ref{3.2.3}.
By induction on $i$, we can suppose that $t_{d_j;\theta_j}$ is described as a word on $X_0$ for $1\leq{j}\leq6$.
Hence we obtain $\psi(t_{d_1;\theta_1}^{\epsilon_1}t_{d_2;\theta_2}^{\epsilon_2}t_{d_3;\theta_3}^{\epsilon_3}t_{d_4;\theta_4}^{\epsilon_4}t_{d_5;\theta_5}^{\epsilon_5}t_{d_6;\theta_6}^{\epsilon_6})=t_{d;\theta}$ for some $\epsilon_j=\pm1$, by the relations (\ref{inverse}) and (\ref{lanterns}).

\begin{figure}[htbp]
\includegraphics{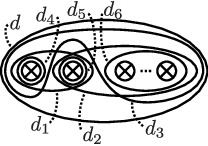}
\caption{}\label{3.2.3}
\end{figure}

\subsubsection{The case where $N_{g,n}\setminus{d}$ is diffeomorphic to $N_{i,1}\sqcup\Sigma_{\frac{g-i}{2},n+1}$}\

When $i=1$, since $t_{d;\theta}=1$ by the relation (\ref{bound}), we obtain $\psi(1)=t_{d;\theta}$.
When $i=2$, $d$ is described as shown in Figure~\ref{3.2.4}, for some model of $N_{g,n}$.
Let $\gamma_1$ and $\gamma_2$ be oriented loops of $N_{g-1,n}$ based at $\ast$ as shown in Figure~\ref{3.2.4}, where $\ast$ is the point obtained by the blowdown with respect to the crosscap $M$ in Figure~\ref{3.2.4}.
Note that $\gamma_1$ and $\gamma_2$ are two sided since $g$ is even when $i=2$.
In addition, when $\CP(\gamma_j)=t_{c_j}t_{c^\prime_j}^{-1}$, $N_{g,n}\setminus(c_j\cup{}c^\prime_j)$ is diffeomorphic to $N_{g,n}\setminus(\alpha_2\cup\zeta)$ for $j=1$ and $2$, where $\alpha_2$ and $\zeta$ appeared in Figure~\ref{generator}.
Hence by Remark~\ref{CP1}, we can describe $\CP(\gamma_j)=x_j(a_2f)x_j^{-1}$ for some $x_j\in\M(N_{g,n})$, for $j=1$ and $2$.
If $x_j\in\T(N_{g,n})$, $x_j$ is represented by a word on $X_0$, and hence $\CP(\gamma_j)$ is a word on $X_0$.
If $x_j\notin\T(N_{g,n})$, $x_jy^{-1}$  is represented by a word on $X_0$.
Since $y(a_2f)y^{-1}=e^{-1}a_2^{-1}$, $\CP(\gamma_j)$ is a word on $X_0$.
We denote the word presentation of $\CP(\gamma_j)$ on $X_0$ by $w_j$.
Then we calculate $\psi(w_2w_1)=w_2w_1\overset{(\ref{t-conj})}{=}\CP(\gamma_2)\CP(\gamma_1)=\CP(\gamma_1\gamma_2)\overset{(\ref{bound})}{=}t_{d;\theta}$, by Remarks~\ref{CP1} and \ref{CP2}.
When $i\geq3$, we take simple closed curves $d_1,\dots,d_6$ as shown in Figure~\ref{3.2.3}, similar to Section~\ref{N_{i,1}-1}.
Since $N_{g,n}\setminus{d_j}$ is diffeomorphic to either $N_{1,1}\sqcup{}N_{g-1,n+1}$, $N_{2,1}\sqcup{}N_{g-2,n+1}$, $N_{i-1,1}\sqcup{}N_{g-i+1,n+1}$ or $N_{i-2,1}\sqcup{}N_{g-i+2,n+1}$, by Section~\ref{N_{i,1}-1}, $t_{d_j;\theta_j}$ is described as a word on $X_0$ for $1\leq{j}\leq6$.
Hence we obtain $\psi(t_{d_1;\theta_1}^{\epsilon_1}t_{d_2;\theta_2}^{\epsilon_2}t_{d_3;\theta_3}^{\epsilon_3}t_{d_4;\theta_4}^{\epsilon_4}t_{d_5;\theta_5}^{\epsilon_5}t_{d_6;\theta_6}^{\epsilon_6})=t_{d;\theta}$ for some $\epsilon_j=\pm1$, by the relations (\ref{inverse}) and (\ref{lanterns}).

\begin{figure}[htbp]
\includegraphics{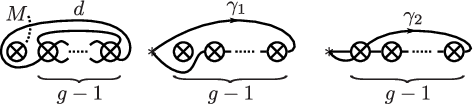}
\caption{}\label{3.2.4}
\end{figure}

Therefore surjectivity of $\psi$ follows, and hence the proof of Proposition~\ref{psi-isom} is completed.

\section{Proof of Theorem~\ref{main-thm} for $g\geq1$ and $n\geq2$}\label{main-thm-high-n}

Recall the capping map $\C$, the point pushing map $\PP$ and the forgetful map $\F$ defined in Section~\ref{CPFB}.
Let $\T^+(N_{g,n-1},\ast)=\C(\T(N_{g,n}))$, then for $g\geq1$ and $n\geq2$, from the sequences (\ref{MCES}) and (\ref{MBES}) in Section~\ref{CPFB}, we have the short exact sequences
\begin{eqnarray}
1\to\Z\to\T(N_{g,n})\overset{\C}{\to}\T^+(N_{g,n-1},\ast)\to1,\label{TCES}\\
1\to\pi_1^+(N_{g,n-1},\ast)\overset{\PP}{\to}\T^+(N_{g,n-1},\ast)\overset{\F}{\to}\T(N_{g,n-1})\to1.\label{TBES}
\end{eqnarray}
Using these sequences and the presentation of $\T(N_{g,1})$ given in Section~\ref{main-thm-low-n}, we give the presentation of $\T(N_{g,n})$ for $g\geq1$ and $n\geq2$, by induction on $n$.

Note that, in general, as basics on combinatorial group theory, from presented groups $\langle{G_1\mid{R_1}}\rangle$ and $\langle{G_3\mid{R_3}}\rangle$ with a short exact sequence $1\to\langle{G_1\mid{R_1}}\rangle\overset{i}{\to}{G}\overset{p}{\to}\langle{G_3\mid{R_3}}\rangle\to1$, we can give a presentation for $G$ as follows.
Let $\tilde{g}$ be any lift of $g\in{G_3}$ with respect to $p$, and $\tilde{r}$ a word obtained from $r\in{R_3}$ by replacing each $g\in{G_3}$ to $\tilde{g}$.
For $x\in\ker{p}$, denote by $w_x$ a word on $i(G_1)$ corresponding to $x$.
For $r\in{R_1}$, denote by $\bar{r}$ a word on $i(G_1)$ obtained from $r$ by replacing $h\in{G_1}$ to $i(h)$.
Let $G_2=\{i(h),\tilde{g}\mid{h\in{G_1},g\in{G_3}}\}$ and $R_2=\{\overline{r_1},\widetilde{r_3}w_{\widetilde{r_3}}^{-1},\tilde{g}i(h)\tilde{g}^{-1}w_{\tilde{g}i(h)\tilde{g}^{-1}}^{-1}\mid{r_1\in{R_1},r_3\in{R_3},h\in{G_1},g\in{G_3}}\}$.
Then we have a presentation $G=\langle{G_2\mid{R_2}}\rangle$.
For details, for instance see Proposition~1 in 10.2 in \cite{J}.

We now show the following proposition, using the sequence (\ref{TBES}).
Remember the infinite presentation $\langle{X\mid{Y}}\rangle$ for the group presented in Theorem~\ref{main-thm}.

\begin{prop}\label{main-thm-ast}
For $g\geq1$ and $n\geq2$, suppose $\T(N_{g,n-1})=\langle{X\mid{Y}}\rangle$,
then $\T^+(N_{g,n-1},\ast)$ admits a presentation with a generating set
$$\widetilde{X}=\left\{t_{\tilde{c},\tilde{\theta}}
\left|
\begin{array}{l}
\tilde{c}~\textrm{is a two sided simple closed curve of}~N_{g,n-1}\setminus\{\ast\}\\
\textrm{which does not bound a disk neighborhood of}~\ast,~\textrm{and}\\
\tilde{\theta}~\textrm{is an orientation of a regular neighborhood of}~\tilde{c}.
\end{array}
\right.
\right\}.$$
The defining relations are
\begin{enumerate}
\item[$(\tilde{1})$]	$t_{\tilde{c},\tilde{\theta}}=1$ if $\tilde{c}$ bounds a disk or a M\"obius band,
\item[$(\tilde{2})$]	$t_{\tilde{c};-_{\tilde{c}}}^{-1}=t_{\tilde{c};+_{\tilde{c}}}$,
\item[$(\tilde{3})$]	all the conjugation relations $ft_{\tilde{c},\tilde{\theta}}f^{-1}=t_{f(\tilde{c});f_\ast(\tilde{\theta})}$ for $f\in\widetilde{X}$,
\item[$(\tilde{4})$]	all the $2$-chain relations,
\item[$(\tilde{5})$]	all the lantern relations,
\item[$(\tilde{5})^\prime$]	all the extended lantern relations defined in Remark~\ref{PP2}.
\end{enumerate}
\end{prop}

\begin{proof}
$\pi_1^+(N_{g,n-1},\ast)$ is a finite rank free group for $n\geq2$.
However, we consider an infinite presentation for $\pi_1^+(N_{g,n-1},\ast)$.
Let $\pi$ be the group generated by symbols $S_\alpha$ for a non-trivial simple loop $\alpha\in\pi_1^+(N_{g,n-1},\ast)$, and with the defining relations
\begin{itemize}
\item	$S_{\alpha^{-1}}=S_\alpha^{-1}$,
\item	$S_{\alpha}S_{\beta}=S_{\gamma}$ if $\alpha\beta=\gamma$,
\item	$S_{\alpha}S_{\beta}S_{\alpha}^{-1}=S_\gamma$ if $\alpha\beta\alpha^{-1}=\gamma$.
\end{itemize}
Then we have that $\pi$ is isomorphic to $\pi_1^+(N_{g,n-1},\ast)$ (see Theorem~1.2 in \cite{K}).
So $\pi$ gives an infinite presentation for $\pi_1^+(N_{g,n-1},\ast)$.
By this we identify $\pi_1^+(N_{g,n-1},\ast)$ with $\pi$.

We take a simple path $P$ of $N_{g,n-1}$ between $\ast$ and the $(n-1)$-st boundary component.
For a simple closed curve $c$ of $N_{g,n-1}$, let $\hat{c}$ be a simple closed curve of $N_{g,n-1}\setminus\{\ast\}$ corresponding to $c$ which does not intersect $P$, as shown in Figure~\ref{lift-1}.

\begin{figure}[htbp]
\includegraphics{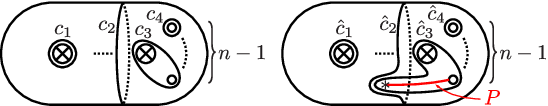}
\caption{Examples of simple closed curves of $N_{g,n-1}\setminus\{\ast\}$ corresponding to simple closed curves of $N_{g,n-1}$ which do not intersect the path $P$.}\label{lift-1}
\end{figure}

By the sequence (\ref{TBES}), $\T^+(N_{g,n-1},\ast)$ is generated by
\begin{itemize}
\item	$t_{\hat{c};\hat{\theta}}$ for any simple closed curve $c$ of $N_{g,n-1}$ and 
\item	$\PP(S_\alpha)$ for any generator $S_\alpha$ of $\pi$.
\end{itemize}
The defining relations are
\begin{itemize}
\item	$\tilde{r}=w_{\tilde{r}}$ for the lift of any $r\in{}Y$ with respect to $\F$,
\item	\begin{itemize}
	\item	$\PP(S_{\alpha^{-1}})=\PP(S_\alpha)^{-1}$,
	\item	$\PP(S_\beta)\PP(S_\alpha)=\PP(S_\gamma)$ for any $S_\alpha$, $S_\beta$ and $S_\gamma$ satisfying $\alpha\beta=\gamma$,
	\item	$\PP(S_\alpha)^{-1}\PP(S_\beta)\PP(S_\alpha)=\PP(S_\gamma)$ for any $S_\alpha$, $S_\beta$ and $S_\gamma$ satisfying $\alpha\beta\alpha^{-1}=\gamma$ and
	\end{itemize}
\item	$t_{\hat{c};\hat{\theta}}\PP(S_\alpha)t_{\hat{c};\hat{\theta}}^{-1}=w_{t_{\hat{c};\hat{\theta}}\PP(S_\alpha)t_{\hat{c};\hat{\theta}}^{-1}}$ for any $t_{\hat{c};\hat{\theta}}$ and $S_\alpha$,
\end{itemize}
where $w_x$ is a word corresponding to $x$ on $\{\PP(S_\alpha)\}$.
We denote this presentation by $\langle{\widetilde{X}_0\mid\widetilde{Y}_0}\rangle$.
In addition, let $\langle{\widetilde{X}\mid\widetilde{Y}}\rangle$ be the infinitely presented group with the presentation given in Proposition~\ref{main-thm-ast}.
We show that $\langle{\widetilde{X}\mid\widetilde{Y}}\rangle$ is isomorphic to $\langle{\widetilde{X}_0\mid\widetilde{Y}_0}\rangle=\T^+(N_{g,n-1},\ast)$.

Denote by $F(\widetilde{X}_0)$ the free group freely generated by $\widetilde{X}_0$.
Let $\widetilde{p}:F(\widetilde{X}_0)\to\langle{\widetilde{X}_0\mid\widetilde{Y}_0}\rangle$ be the natural projection and $\widetilde{\eta}:F(\widetilde{X}_0)\to\langle{\widetilde{X}\mid\widetilde{Y}}\rangle$ the homomorphism defined as $\widetilde{\eta}(x)=x$ for any $x\in\widetilde{X}_0$.
Note that, by Remark~\ref{PP1}, $\PP(S_\alpha)$ is also in $\langle{\widetilde{X}\mid\widetilde{Y}}\rangle$ for any generator $S_\alpha$ of $\pi$.
Hence $\widetilde{\eta}$ is well-defined.
We consider a correspondence
$$\widetilde{\psi}:\langle{\widetilde{X}_0\mid\widetilde{Y}_0}\rangle\to\langle{\widetilde{X}\mid\widetilde{Y}}\rangle$$
satisfying $\widetilde{\psi}\circ\widetilde{p}=\widetilde{\eta}$.
We show that $\widetilde{\psi}$ is an isomorphism.

First we show well-definedness of $\widetilde{\psi}$.
In the first relation of $\widetilde{Y}_0$, from the definition of $\hat{c}$, it is clear that $w_{\tilde{r}}=1$ for any $r\in{}Y$.
Hence the first relation of $\widetilde{Y}_0$ is either one of the relations~$(\tilde{1})$-$(\tilde{5})$.
In the second relation of $\widetilde{Y}_0$, the relation $\PP(S_{\alpha^{-1}})=\PP(S_\alpha)^{-1}$ is trivial.
The relation $\PP(S_\beta)\PP(S_\alpha)=\PP(S_\gamma)$ is obtained from the relations~$(\tilde{3})$ or $(\tilde{5})^\prime$ by Remark~\ref{PP2}.
The relation $\PP(S_\alpha)^{-1}\PP(S_\beta)\PP(S_\alpha)=\PP(S_\gamma)$ is the relation~$(\tilde{3})$.
In the third relation of $\widetilde{Y}_0$, we have $w_{t_{\tilde{c};\tilde{\theta}}\PP(S_\alpha)t_{\tilde{c};\tilde{\theta}}^{-1}}=\PP(S_{t_{\tilde{c};\tilde{\theta}}(\alpha)})$, and hence it is the relation~$(\tilde{3})$.
Therefore, any relation of $\widetilde{Y}_0$ is satisfied in $\langle{\widetilde{X}\mid\widetilde{Y}}\rangle$.
So $\widetilde{\psi}$ is well-defined as a homomorphism.

Next we show bijectivity of $\widetilde{\psi}$.
Let $\widetilde{\varphi}:\langle{\widetilde{X}\mid\widetilde{Y}}\rangle\to\langle{\widetilde{X}_0\mid\widetilde{Y}_0}\rangle=\T^+(N_{g,n-1},\ast)$ be the homomorphism defined as $\widetilde{\varphi}(x)=x$ for any $x\in\widetilde{X}$.
Since $\widetilde{\varphi}(\widetilde{Y})=1$ in $\T^+(N_{g,n-1},\ast)$ clearly, $\widetilde{\varphi}$ is well-defined.
By the definitions of $\widetilde{\psi}$ and $\widetilde{\varphi}$, it is clear that $\widetilde{\varphi}\circ\widetilde{\psi}$ is the identity map, and hence $\widetilde{\psi}$ is injective.
For any $t_{\tilde{c},\tilde{\theta}}\in\widetilde{X}$, if $\tilde{c}$ intersects $P$ transversally at $l\geq1$ points, there exist $t_{\tilde{c}^\prime,\tilde{\theta}^\prime}\in\widetilde{X}$ and $S_\alpha$ such that $t_{\tilde{c},\tilde{\theta}}=\widetilde{\psi}(\PP(S_\alpha))t_{\tilde{c}^\prime,\tilde{\theta}^\prime}$, as shown in Figure~\ref{lift-2}, by Remark~\ref{CP1}.
We notice that $\tilde{c}^\prime$ intersects $P$ transversally at $l-1$ points.
By induction on the intersection number $l$ of $\tilde{c}$ and $P$, we see that $t_{\tilde{c},\tilde{\theta}}$ is a product of some $\widetilde{\psi}(t_{\hat{c};\hat{\theta}})$ and some $\widetilde{\psi}(\PP(S_\alpha))$'s.
Hence $\widetilde{\psi}$ is surjective, and so bijective.

\begin{figure}[htbp]
\includegraphics{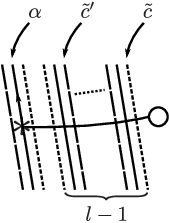}
\caption{$\tilde{c}$ and $\tilde{c}^\prime$ are the boundary curves of a regular neighborhood of $\alpha$.}\label{lift-2}
\end{figure}

Therefore $\widetilde{\psi}$ is the isomorphism.
Thus the claim is obtained.
\end{proof}

Finally we complete the proof of Theorem~\ref{main-thm}.

\begin{proof}[Proof of Theorem~\ref{main-thm}]
From the case where $n=1$ of Theorem~\ref{main-thm}, Proposition~\ref{main-thm-ast} and the sequence (\ref{TCES}), by induction on $n$, $\T(N_{g,n})$ is generated by
\begin{itemize}
\item	the natural lift $t_{c;\theta}$ of $t_{\tilde{c};\tilde{\theta}}$ with respect to $\C$ and
\item	the Dehn twist $t_{\partial_n;\theta_n}$ about the $n$-th boundary curve $\partial_n$,
\end{itemize}
that is, $X$ generate $\T(N_{g,n})$.
In addition $\T(N_{g,n})$ has the relations
\begin{itemize}
\item	$\tilde{r}=w_{\tilde{r}}$ for the lift of any relator $r$ of $\T^+(N_{g,n-1},\ast)$ with respect to $\C$ and
\item	$t_{c;\theta}t_{\partial_n;\theta_n}t_{c;\theta}^{-1}=w_{t_{c;\theta}t_{\partial_n;\theta_n}t_{c;\theta}^{-1}}$,
\end{itemize}
where $w_x=t_{\partial_n;\theta_n}^m$ for some integer $m$ corresponding to $x$.
In the first relation, it is clear that $w_{\tilde{r}}=1$ if $r$ is a relator corresponding to the relations~$(\tilde{1})$-$(\tilde{5})$ of Proposition~\ref{main-thm-ast}.
On the other hand, if $r$ is a relator corresponding to the relation~$(\tilde{5})^\prime$ of Proposition~\ref{main-thm-ast}, then $w_{\tilde{r}}=t_{\partial_n;\theta_n}^\epsilon$ for some $\epsilon=\pm1$.
Hence the first relation is either one of the relation~(\ref{bound})-(\ref{lanterns}).
In the second relation, it is clear that $w_{t_{c;\theta}t_{\partial_n;\theta_n}t_{c;\theta}^{-1}}=t_{\partial_n;\theta_n}$, and hence it is the relation~(\ref{t-conj}).

Thus we complete the proof.
\end{proof}

\section*{Acknowledgement}

The authors would like to express their gratitude to Susumu Hirose for his useful comments.


\end{document}